\documentclass[a4paper, 11pt, reqno]{amsart}
\usepackage[utf8]{inputenc}
\usepackage[usenames,dvipsnames]{color}
\usepackage{amssymb, amsmath, latexsym, graphics, graphicx}
\usepackage[margin=3.5cm]{geometry}
\usepackage[linesnumbered,ruled,vlined]{algorithm2e}
    \DontPrintSemicolon
    \SetAlgoSkip{medskip}

\newcommand{\ol}{\overline}

\usepackage{amsfonts}
\usepackage{amsthm}
\usepackage{enumerate}
\usepackage{hyperref}

\newtheorem{theorem}{Theorem}[section]
\newtheorem{corollary}[theorem]{Corollary}
\newtheorem{lemma}[theorem]{Lemma}
\newtheorem{proposition}[theorem]{Proposition}
\newtheorem{question}[theorem]{Question}

\newtheorem{remark}[theorem]{Remark}
\theoremstyle{definition}

\newtheorem{example}[theorem]{Example}

\newcommand{\N}{\mathbb{N}}

\newcommand{\cM}{\mathcal{M}}
\newcommand{\cB}{\mathcal{B}}

\newcommand{\supp}{\mathrm{supp}}

\newcommand{\pref}{\mathrm{pref}}
\newcommand{\suff}{\mathrm{suff}}
\newcommand{\flad}{\mathrm{FLAd}}
\newcommand{\frad}{\mathrm{FRAd}}
\newcommand{\fad}{\mathrm{FAd}}
\newcommand{\fm}{\mathrm{F}}

\newcommand{\fu}{\mathrm{F}_u}
\newcommand{\fbu}{\mathrm{F}_{bu}}

\newcommand{\ut}{\mathrm{UT}}
\newcommand{\fmp}{(\mathrm{F}(\Sigma))^+}
\newcommand{\osigma}{\overline{\Sigma}}
\newcommand{\maxp}{\mathrm{mp}}
\newcommand{\setup}{\GraphInit[vstyle=Empty]
            \SetVertexSimple[MinSize = 1pt]
            \SetUpEdge[lw = 0.5pt]
            \tikzset{EdgeStyle/.style={->-}}
            \tikzset{VertexStyle/.append style = {minimum size = 3pt, inner sep = 0pt}}
            \SetVertexNoLabel
            \SetGraphUnit{2}}

\newcommand{\floor}[1]{\left \lfloor #1 \right \rfloor }

\usepackage{tkz-graph}\usetikzlibrary{decorations.markings}\usetikzlibrary{quotes}
\tikzset{->-/.style={
 decoration={
   markings,
   mark=at position #1 with {\arrow{>}}
   },
   postaction={decorate}
  },
  ->-/.default=0.5 
}

\def\PlusCross{{%
    \setbox0\hbox{$+$}%
    \rlap{\hbox to \wd0{\hss$\times$\hss}}\box0
}}

\usetikzlibrary{matrix}
\usetikzlibrary{decorations.pathreplacing}

\title{Growth and identities of monogenic free adequate monoids}

\author{Thomas Aird}
\address{T. Aird, Department of Mathematics, University of Manchester, 
Manchester \linebreak M13 9PL and Heilbronn Institute for Mathematical Research, Bristol, UK.}
\email{thomas.aird@manchester.ac.uk}

\author{Daniel Heath}
\address{D. Heath, Department of Mathematics, University of Manchester, 
Manchester \linebreak M13 9PL}
\email{daniel.heath-2@manchester.ac.uk}
\keywords{adequate, Ehresmann, growth, identities, monoid. \\
\indent \emph{MSC(2020)}: 20M05 (Primary), 08B20, 20M07, 05C05}

\date{September 2025}

\begin{document}
\begin{abstract}
    Motivated by recent advances in inverse semigroup theory, we investigate the growth of and identities satisfied by free left and free two-sided adequate monoids.

    We explicitly compute the growth of the monogenic free left adequate monoid with the usual unary monoid generating set and show it has intermediate growth owing to a connection with integer partitions. In the two-sided case, we establish a lower bound on the (idempotent) growth rate of the monogenic free adequate monoid, showing that it grows exponentially.

    We completely classify the enriched identities satisfied by the monogenic free left adequate monoid and deduce that it satisfies the same monoid identities as the sylvester monoid. In contrast, we show that the monogenic free two-sided adequate monoid satisfies no non-trivial monoid identities.
\end{abstract}
\maketitle
\section{Introduction}
First introduced by Fountain \cite{FOU, FOU2}, the class of (\textit{left}) \textit{adequate semigroups} generalises the well-studied class of \textit{inverse semigroups}.
They form part of the \textit{York school} of semigroup theory, which (broadly) considers classes of semigroups whose structure is strongly determined by their idempotents. Indeed, the nomenclature is due to adequate semigroups containing a sufficient quantity of idempotents \cite{FOU2}.

It is common \cite{BRA2,HEA,KAM,KAM2} to consider left and right adequate monoids within the class of $(2,1,0)$-algebras, with an additional unary operation relating elements to idempotents with whom they share certain cancellativity properties (see Section \ref{sec:prelims:adequate}). Within this class, both left and right adequate monoids form \textit{quasivarieties} and as such free objects exist for any rank.

The free inverse monoid was combinatorially described by Munn \cite{MUN}, building on work of Scheiblich \cite{SCH} via birooted, edge-labelled directed graphs with a multiplication given by a `gluing and folding' procedure. This has prompted far-reaching study of inverse monoids in a wide variety of natural directions (see, for example, \cite{JON,KAM3,LAW,STE}).

Such descriptions for free left, free right, and free two-sided adequate monoids have also been given by Kambites \cite{KAM,KAM2} with \textit{retractions} providing the generalisation of folding. This modification of folding adds much complexity to the study of adequate monoids, even in the free cases.

In the \textit{monogenic} cases, however, retractions are (somewhat) more tractable. Here, we explore $\flad(a)$, $\frad(a)$, and $\fad(a)$, and study various properties motivated by recent advances in inverse semigroup theory, particularly results akin to those for free inverse monoids.


It is well known (and easily seen) that the \textit{spherical growth} of $\mathrm{FIM}(a)$ is quadratic. Such growth for higher rank free inverse monoids is seen to be exponential, and has been studied in detail in \cite{KAM3}. We examine here the growth rate of $\flad(a)$ and $\fad(a)$, and show that $\flad(a)$ has intermediate growth, whilst in contrast, even the semilattice of idempotents of $\fad(a)$ grows exponentially. In particular, we deduce that $\fad(X)$ grows exponentially for any $|X| \geq 1$.

We then turn our attention to the identities satisfied by $\flad(a)$ and $\frad(a)$ in both enriched and unenriched signatures. It is well known that the monogenic inverse monoid satisfies the same monoid identities as many other interesting semigroups including; the \textit{bicyclic monoid} \cite{SInverseBicyclic}, the monoid of $2 \times 2$ \textit{upper triangular tropical matrices} \cite{DJKUpperTriTrop}, and the \textit{plactic monoid} of rank $2$ \cite{CKKMOPlactic,IPlactic}. 
We show here similar results for $\flad(a)$ and $\frad(a)$  corresponding to the \textit{sylvester} and the \textit{\#-sylvester} monoids, respectively. 

We proceed as follows. In Section~\ref{sec:prelims}, we gather the necessary preliminaries for our forthcoming results. In Section~\ref{sec:PropsOfMonAde}, we collect results which hold specifically in the monogenic case -- we show a certain identity of Batbedat \cite{BAT} and Fountain \cite{FOU5} holds in $\flad(a)$. In Section~\ref{sec:growth}, we explore the growth rates of $\flad(a)$ and $\fad(a)$. Finally, in Section~\ref{sec:identities}, we use the identities of Section \ref{sec:PropsOfMonAde} to determine both the enriched and monoid identities satisfied by $\flad(a)$, $\frad(a)$, and $\fad(a)$. We also comment on higher rank cases.

\textbf{Acknowledgements.} This work was supported by the Heilbronn Institute for Mathematical Research. The authors thank James Bryden for their help in verifying the results in Table \ref{tab:sizeoffad}.

\section{Preliminaries}\label{sec:prelims}

Throughout, we assume the reader is familiar with the fundamentals of semigroup theory and universal algebra. For comprehensive introductions, we direct the reader to \cite{HOW} for concepts in semigroup theory, and to \cite{BUR} for concepts in universal algebra.

\subsection{Adequate monoids}\label{sec:prelims:adequate}

Left [resp. right] adequate monoids were defined traditionally via a relation $\mathcal{R}^*$ [resp. $\mathcal{L}^*$], see for example \cite{FOU,FOU2}. Here, we take an alternative approach via \textit{identities}, as seen in \cite{BRA,BRA2}. We will discuss identities in more detail in Section \ref{sec:identities}.

A monoid $M$ equipped with an additional unary operation denoted $m \mapsto m^+$, is called \textit{left adequate} if it satisfies the quasi-identities 
\[x^+x\approx x \quad\quad (x^+y^+)^+ \approx x^+y^+ \approx y^+x^+ \quad\quad (xy)^+ \approx (xy^+)^+\]\[x^2 \approx x \rightarrow x \approx x^+ \quad\quad xz \approx yz \rightarrow xz^+ \approx yz^+.\]

A monoid $M$ equipped with an additional unary operation denoted $m \mapsto m^*$, is called \textit{right adequate} if it satisfies the quasi-identities 
\[xx^*\approx x \quad\quad (x^*y^*)^* \approx x^*y^* \approx y^*x^* \quad\quad (xy)^* \approx (x^*y)^*\]\[x^2 \approx x \rightarrow x \approx x^* \quad\quad zx \approx zy \rightarrow z^*x \approx z^*y.\]

A monoid $M$ equipped with two additional unary operations $+$ and $\ast$ is called \textit{two-sided adequate}, or simply \textit{adequate}, if it is left adequate with respect to $+$ and right adequate with respect to $\ast$.



Unless otherwise stated, we consider left [resp. right] adequate monoids as algebras of signature {$(2,1,0)$} with the associative multiplication and identity element supplemented with the unary operation $+$ [resp. $\ast$]. We similarly treat adequate monoids within the signature $(2,1,1,0)$.


\subsection{Free structures}

Given a class of algebras $\mathcal{C}$ within a fixed signature, we say an object $F$ is \textit{free} on a subset $X \subseteq F$ if for any $M \in \mathcal{C}$ and map $f:X \to M$, there exists a unique morphism $F \to M$ extending $f$. The set $X$ is then said to \textit{freely generate} $F$, and its cardinality is the \textit{rank} of $F$.

In \textit{varieties} and \textit{quasivarieties}, free objects exist in any rank. We will consider various such classes here. Given a set $X$, the \textit{free unary monoid} [resp. the \textit{free biunary monoid}] generated by $X$, denoted $\fu(X)$ [resp. $\fbu(X)$] is the free structure freely generated by $X$ in the class of $(2,1,0)$-algebras [resp. $(2,1,1,0)$-algebras] with defining quasi-identities only those of monoids, i.e.\[x(yz)\approx (xy)z \quad\quad x\varepsilon \approx x \approx \varepsilon x.\]Note that in both cases, the unary operation(s) are entirely free and we do not have $\varepsilon^+ = \varepsilon$ as one might expect. As we shall see, we do have this equality in $\flad(X)$ and $\fad(X)$.

In their respective signatures, the classes of left adequate, right adequate, and adequate monoids each form quasivarieties (see also \cite[Section 2.2]{HEA}) and as such free objects exist for any rank. For a set $X$, we denote by:
\begin{itemize}
    \item $\flad(X)$, the free left adequate monoid freely generated by $X$;
    \item $\frad(X)$, the free right adequate monoid freely generated by $X$;
    \item $\fad(X)$, the free adequate monoid freely generated by $X$.
\end{itemize}These free structures coincide with the \textit{free left Ehresmann}, \textit{free right Ehresmann} and \textit{free Ehresmann} monoids respectively \cite{GOU2,KAM2}.

We state the following, which is immediately apparent from the definition of $\flad(X)$ and $\frad(X)$.

\begin{proposition}\label{prop:ladradduality}
    For any set $X$, $\flad(X)$ and $\frad(X)$ are anti-isomorphic in the $(2,1,0)$-signature.
\end{proposition}

These structures were given a geometric description in \cite{KAM,KAM2} which we recall in Section \ref{sec:prelims:trees}.

\subsection{Trees}\label{sec:prelims:trees}

Given a (non-empty) set $X$, a \textit{$X$-tree} $\Gamma$ is a directed graph, with edges labelled by elements of $X$, whose underlying, undirected graph is a tree, with two distinguished vertices called \textit{start} and \textit{end} such that there is a (necessarily unique) directed path from the start vertex to the end vertex. We call this path the \textit{trunk} of the tree, with its vertices and edges called \textit{trunk vertices} and \textit{trunk edges} respectively. We draw $X$-graphs via the usual diagrams (see Figure \ref{fig:tuples} for example), with the symbol $+$ decorating the start vertex and $\times$ decorating the end vertex.

We write $|\Gamma|$ for the number of edges of $\Gamma$. We denote the $X$-tree with no edges (and coinciding start and end vertex) by $\varepsilon$.

For $\Gamma$ with $k$ trunk edges, we write the $k+1$ trunk vertices of $\Gamma$ as a tuple $(v_0,\dots,v_{k})$, where $v_0$ is the start vertex, $v_{k}$ is the end vertex, and there is a trunk edge from $v_i$ to $v_{i+1}$ for all $0 \leq i < k$. We say there exists a \textit{branch} at a trunk vertex $v_i$ if there exists some non-trunk edge of $\Gamma$ with initial vertex $v_i$.

We will often be interested in the case where $X$ is a singleton $\{a\}$ -- in this case, we write \textit{$a$-tree} in place of $X$-tree.

We call an $X$-tree $\Gamma$ a \textit{left $X$-tree} if there is a directed path from the start vertex to every vertex, and dually a \textit{right $X$-tree} if there is a directed path from every vertex to the end vertex. 

A \textit{morphism} of $X$-graphs is a directed graph morphism which respects labelling and sends the start vertex to the start vertex and the end vertex to the end vertex. A \textit{retraction} on an $X$-tree $\Gamma$ is a directed graph endomorphism $\Gamma \to \Gamma$ which is idempotent. Its image is called a \textit{retract}. Following the conventions in \cite{HEA}, we call a tree \textit{retract-free} if it admit no non-trivial retractions -- it follows from general facts on relational structures (see e.g. \cite[Proposition 3.5]{KAM} or \cite[Proposition 1.4.7]{HEL}) that every $X$-tree $\Gamma$ admits a unique (up to directed graph isomorphism) retract-free retract, which we denote by $\overline{\Gamma}$. We refer to the process $\Gamma \mapsto \overline{\Gamma}$ by \textit{retracting}.

\begin{example}
    Figure \ref{fig:tuples} shows three left $a$-trees. The leftmost and rightmost trees are retract-free. The central tree admits a retraction to the leftmost tree. The leftmost tree has one branch. The central and rightmost trees each have two branches.

    \begin{figure}[h]
        \centering
        \begin{tikzpicture}
            \GraphInit[vstyle=Empty]
            \SetVertexSimple[MinSize = 1pt]
            \SetUpEdge[lw = 0.5pt]
            \tikzset{EdgeStyle/.style={->-}}
            \tikzset{VertexStyle/.append style = {minimum size = 3pt, inner sep = 0pt}}
            \SetVertexNoLabel
            \SetGraphUnit{2}

            \node (A) at ( 0,0) {\large$+$};
            \Vertex[x=0,y=1]{C1}
            \Vertex[x=1,y=1]{L1}
            \Vertex[x=2,y=1]{L2}
            \node (R1) at ( 0,2) {\large$\times$};
                      
            \Edge(A)(C1)\draw (A) -- (C1) node [midway, left=2pt] {$a$};
            \Edge(C1)(R1)\draw (C1) -- (R1) node [midway, left=2pt] {$a$};
            \Edge(C1)(L1)\draw (C1) -- (L1) node [midway, above=2pt] {$a$};
            \Edge(L1)(L2)\draw (L1) -- (L2) node [midway, above=2pt] {$a$};
            
        \end{tikzpicture}
        \hspace{0.05\textwidth}
        \begin{tikzpicture}
            \GraphInit[vstyle=Empty]
            \SetVertexSimple[MinSize = 1pt]
            \SetUpEdge[lw = 0.5pt]
            \tikzset{EdgeStyle/.style={->-}}
            \tikzset{VertexStyle/.append style = {minimum size = 3pt, inner sep = 0pt}}
            \SetVertexNoLabel
            \SetGraphUnit{2}
          
            \node (A) at ( 0,0) {\large$+$};
            \Vertex[x=0,y=1]{C1}
            \Vertex[x=1,y=1]{L1}
            \Vertex[x=2,y=2]{L2}
            \node (R1) at ( 0,2) {\large$\times$};
            \Vertex[x=1,y=0]{E1}
            \Vertex[x=2,y=1]{E2}

            \Edge(A)(C1)\draw (A) -- (C1) node [midway, left=2pt] {$a$};
            \Edge(C1)(R1)\draw (C1) -- (R1) node [midway, left=2pt] {$a$};
            \Edge(C1)(L1)\draw (C1) -- (L1) node [midway, above=2pt] {$a$};
            \Edge(L1)(L2)\draw (L1) -- (L2) node [midway, left=2pt] {$a$};
            \Edge(A)(E1)\draw (A) -- (E1) node [midway, above=2pt] {$a$};
            \Edge(L1)(E2)\draw (L1) -- (E2) node [midway, below=2pt] {$a$};
            
        \end{tikzpicture}
        \hspace{0.05\textwidth}
        \begin{tikzpicture}
            \GraphInit[vstyle=Empty]
            \SetVertexSimple[MinSize = 1pt]
            \SetUpEdge[lw = 0.5pt]
            \tikzset{EdgeStyle/.style={->-}}
            \tikzset{VertexStyle/.append style = {minimum size = 3pt, inner sep = 0pt}}
            \SetVertexNoLabel
            \SetGraphUnit{2}

            \node (_) at ( 0,0) {}; 
            
            \node (A) at ( 0,0.5) {\large$+$};
            \Vertex[x=2,y=1.5]{R2}
            \Vertex[x=1,y=0.5]{L1}
            \Vertex[x=2,y=0.5]{L2}
            \Vertex[x=3,y=0.5]{L3}
            \Vertex[x=4,y=0.5]{L4}
            \node (C1) at ( 0,1.5) {\large$\times$};
            \Vertex[x=1,y=1.5]{R1}
                      
            \Edge(A)(C1)\draw (A) -- (C1) node [midway, left=2pt] {$a$};
            \Edge(C1)(R1)\draw (C1) -- (R1) node [midway, above=2pt] {$a$};
            \Edge(R1)(R2)\draw (R1) -- (R2) node [midway, above=2pt] {$a$};
            \Edge(A)(L1)\draw (A) -- (L1) node [midway, above=2pt] {$a$};
            \Edge(L1)(L2)\draw (L1) -- (L2) node [midway, above=2pt] {$a$};
            \Edge(L2)(L3)\draw (L2) -- (L3) node [midway, above=2pt] {$a$};
            \Edge(L3)(L4)\draw (L3) -- (L4) node [midway, above=2pt] {$a$};
            
        \end{tikzpicture}
        \caption{Some left $a$-trees.}
        \label{fig:tuples}
    \end{figure}
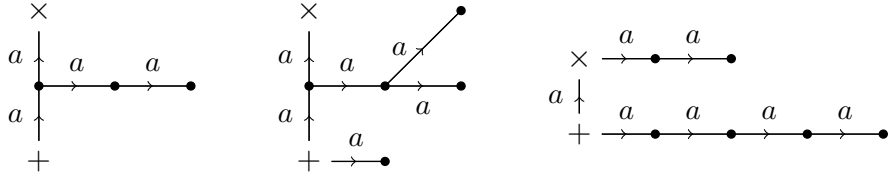
\end{example}

Throughout, we work with $X$-trees up to birooted graph isomorphism, and for notational simplicity, we will (where no possible ambiguity arises) identify $X$-trees with their isomorphism types. The following results summarise the main results of \cite{KAM,KAM2}. 

\begin{theorem}[Kambites \cite{KAM,KAM2}]\label{thm:ut}
    The set of all $X$-trees $\ut^1(X)$ forms an \linebreak $X$-generated biunary monoid with multiplication $S\times T$ given by gluing $T$ to $S$ start-to-end, and taking start vertex the start vertex of $S$ and end vertex the end vertex of $T$. The two unary operations are: $+$, with $T^+$ given by moving the end vertex to the start vertex, and $\ast$, with $T^\ast$ given by moving the start vertex of $T$ to the end vertex.
    
    The set of all left \emph{[}resp. right\emph{]} $X$-trees, denoted $\mathrm{LUT}^1(X)$ \emph{[}resp. $\mathrm{RUT}^1(X)$\emph{]}, forms an $X$-generated unary submonoid of $\ut^1(X)$ with the operation $+$ \emph{[}resp. $\ast$\emph{]}.
\end{theorem}

\begin{theorem}[Kambites \cite{KAM,KAM2}]\label{thm:flad}
    The free left \emph{[}resp. right\emph{]} adequate monoid $\flad(X)$ \emph{[}resp. $\frad(X)$\emph{]} generated by a set $X$ is the unary monoid quotient of $\mathrm{LUT}^1(X)$ with $+$ \emph{[}resp. $\mathrm{RUT}^1(X)$ with $\ast$\emph{]} induced by the retraction map $T \mapsto \overline{T}$.
    
    The free adequate monoid $\fad(X)$ generated by a set $X$ is the biunary monoid quotient of $\ut^1(X)$ induced by the retraction map $T \mapsto \overline{T}$.
    
    These structures are freely generated (in their respective signatures) by the set of single-edge trees with distinct start and end vertices labelled $x$ for each $x \in X$ (we identify this set of trees with $X$ itself).
\end{theorem}

We will later require the following result, also from \cite{KAM}, which we state here.

\begin{proposition}[{\cite[Proposition 4.8]{KAM}}]\label{prop:idems}
    Let $T$ be an $X$-tree. The following are equivalent:
    \begin{enumerate}
        \item $T$ has no trunk edges,
        \item $\overline{T}$ has no trunk edges,
        \item The start vertex and end vertex of $T$ coincide,
        \item The start vertex and end vertex of $\overline{T}$ coincide,
        \item $\overline{T} \in \fad(X)$ is idempotent, that is $\overline{T}^2 = \overline{T}^+ = \overline{T}$.
    \end{enumerate}
\end{proposition}

\subsection{Partitions}

Given non-negative integers $n$ and $k$, we let $P(n)$ be the number of partitions of $n$, and $P(n,k)$ be the number of partitions of $n$ into $k$ parts. Similarly, let $Q(n)$ be the number of partitions of $n$ only using \textit{distinct} numbers, and $Q(n,k)$ be the number of of partitions of $n$ into $k$ \textit{distinct} parts. Note the following folklore lemma.
\begin{lemma} \label{lem:FolklorePartitions}
    Let $n, k \in \N$. Then, 
    \begin{itemize}
        \item $P(n,k) = P(n-1,k-1) + P(n-k,k)$,
        \item $Q(n,k) = Q(n-k,k-1) + Q(n-k,k)$,
        \item $P(n,k) = Q(n + c_k,k)$ where $c_k = \binom{k}{2}$.
    \end{itemize}
\end{lemma}
Throughout, we write $[k] = \{0,\dots,k\}$ for non-negative $k$. As with trunk vertices of $X$-trees, we index general $(k+1)$-tuples by $[k]$, that is, we write such a tuple as $\sigma = (\sigma_0,\dots,\sigma_k)$.

\section{Properties of monogenic free left adequate monoids} \label{sec:PropsOfMonAde}

We collect here some properties of monogenic free left adequate monoids. Let $T$ be a left $a$-tree. Since $E(\flad(a)) \hookrightarrow E(\fad(a))$, it follows from Proposition \ref{prop:idems} that $\overline{T} \in \flad(a)$ is idempotent if and only if the start vertex and end vertex of $T$ coincide. For any such $T$, there certainly exists a retraction from $T$ to a longest path in $T$ -- we have shown the following. 

\begin{proposition}\label{prop:idemsofmono}
     $E(\flad(a)) = \left\{(a^j)^+ \colon j \in \N_0\right\}$.
\end{proposition}

\begin{corollary}\label{cor:productofmonoidems}
    For idempotents $e,f \in E(\flad(a))$, $ef = \left(a^{\max(|e|,|f|)}\right)^+$.
\end{corollary}

\begin{proof}
    By Proposition \ref{prop:idemsofmono}, $e = \left(a^{|e|}\right)^+$ and $f = \left(a^{|f|}\right)^+$. Clearly then 
    \[ef = \left(a^{|e|}\right)^+\left(a^{|f|}\right)^+ = \left(a^{\max(|e|,|f|)}\right)^+\]
    as the shorter branch of $\left(a^{|e|}\right)^+\times \left(a^{|f|}\right)^+$ may retract onto the longer branch.
\end{proof}

\begin{lemma}\label{lem:fladnonnested}
    Let $x,y,z \in \flad(a)$. Then we have $\varepsilon^+ = \varepsilon$, $(x^+)^+ = x^+$, and $(xy^+z)^+ = (xy)^+(xz)^+$.
\end{lemma}
\begin{proof}
    The equalities $\varepsilon^+ = \varepsilon$ and $(x^+)^+ = x^+$ are immediate from the defining quasi-identities, and it remains to show $(xy^+z)^+ = (xy)^+(xz)^+$.

    By Proposition~\ref{prop:idemsofmono}, note that, for any $w \in \flad(a)$, $w^+$ is the idempotent of the same length as a longest path in $w$ as all other paths can retract into the longest path.

    Thus, $(xy^+z)^+$ is either equal to $x^+$, $(xy)^+$, or $(xz)^+$ depending on if the longest path is contained entirely in $x$, $xy$, or $xz$ respectively. 
    Moreover, by Corollary \ref{cor:productofmonoidems}, the product of two idempotents is equal to the longer idempotent, and so\linebreak
    $(xy^+z)^+ = x^+(xy)^+(xz)^+$.

    Finally, note that $(xy)^+$ is always at least as long as $x^+$, so 
    \[(xy^+z)^+ = x^+(xy)^+(xz)^+ = (xy)^+(xz)^+. \qedhere\]
\end{proof}

The third identity of Lemma \ref{lem:fladnonnested} appears in \cite{BAT} and its $\ast$-dual in \cite{FOU5}, where it defines \textit{right h-adequate monoids}. It follows that $\flad(a)$ is also the monogenic free \textit{left h-adequate monoid}. It is seen in \cite{FOU5} that the monogenic free left h-adequate monoid (and hence $\flad(a)$) is \textit{residually finite} and \textit{hopfian}, and may be given as a quotient of $(\mathbb{N},+) \ast (\mathbb{N},\max)$ with the alternating integers of an element describing segments of the trunk and (longest) branch lengths.

\begin{lemma} \label{lem:removebranches}
    Let $x,y \in \flad(a)$ and $\Theta(x)$ be the subtree of $x$ exactly consisting of the trunk edges. Then, $xyx = \Theta(x)yx$.
\end{lemma}
\begin{proof}
    The monoid $\flad(a)$ is generated by $a$ and $E(\flad(a))$ as a monoid. It follows from Proposition \ref{prop:idemsofmono} that we may write $x$ and $y$ in the following forms: 
    \[
        x = (a^{n_0})^+\prod_{i=1}^s a(a^{n_i})^+ \text{ and } y = (a^{m_0})^+\prod_{i=1}^t a(a^{m_i})^+
    \]for some $s,t,n_i,m_i \in \mathbb{N}_0$. Then, we obtain that $xyx$ is given by
    \begin{equation*}
        xyx = a^s(a^{m_0})^+\left(\prod_{i=1}^t a(a^{m_i})^+\right)(a^{n_0})^+\left(\prod_{i=1}^s a(a^{n_i})^+\right) = \Theta(x)yx
    \end{equation*}as the leftmost occurrence of any $(a^{n_i})^+$ can be retracted into the path consisting of a portion of the trunk and the branch of the rightmost copy of $(a^{n_i})^+$.
\end{proof}

\section{Growth of monogenic free adequate monoids}\label{sec:growth}

We consider here the growth of $\flad(a)$. By appealing to Proposition \ref{prop:ladradduality}, dual results will apply to $\frad(a)$. Recall that we consider left adequate monoids to be $X$-generated as unary monoids, and free adequate monoids to be $X$-generated as biunary monoids. 

\subsection{Sphere growth in adequate monoids}

Given a word $w$ in the free $X$-generated unary monoid $\fu(X)$ or the free $X$-generated biunary monoid $\fbu(X)$, the \textit{length} of $w$, denoted $|w|$, is the number appearances of letters of $X$ in $w$.

\begin{example}
    In $\fu(a)$, $|a\varepsilon^+(aa^+)^+| = 3$. In $\fu(\{a,b\})$, $|(ab)^+b^+\varepsilon^+(a^+)^+| = 4$.
\end{example}

For $n \in \N_0$, the \textit{ball of size $n$ in $\flad(X)$}, denoted $B_L^X(n)$, is the set of trees $T \in \flad(X)$ such that there exists some $(2,1,0)$-word $w \in \fu(X)$ with $|w| \leq n$ and $T = \psi(w)$ where $\psi$ is the canonical surjective morphism $\fu(X) \twoheadrightarrow \flad(X)$. The \textit{sphere of size $n$ in $\flad(X)$}, denoted $S_L^X(n)$, is the set \[S_L^X(n) := B_L^X(n) \setminus B_L^X(n-1).\]

When $X$ is implicitly understood, we suppress our superscript notation and write simply $B_L(n)$ and $S_L(n)$.

We analogously define: \begin{itemize}
    \item The ball of size $n$, $B_{\ut}(n)$, in $\ut^1(X)$ via the canonical surjective morphism $\fbu(X) \twoheadrightarrow \ut^1(X)$.
    \item The ball of size $n$, $B_{\mathrm{LUT}}(n)$, in $\mathrm{LUT}^1(X)$ via the canonical surjective morphism $\fu(X) \twoheadrightarrow \mathrm{LUT}^1(X)$.
\end{itemize}
In $\fad(X)$, we also use the following similar notation and define:
\begin{itemize}
    \item The ball of size $n$, $B(n)$, in $\fad(X)$ via the canonical surjective morphism $\fbu(X) \twoheadrightarrow \fad(X)$.
    \item The sphere of size $n$, $S(n)$, in $\fad(X)$ to be $S(n) := B(n) \setminus B(n-1)$.
\end{itemize}

We now study balls and spheres, and show that we can interpret these on the level of $X$-trees from Section \ref{sec:prelims:trees}.

\begin{lemma}
    Let $n \in \N_0$. \begin{enumerate}[(i)]
        \item Let $T \in \ut^1(X)$. Then $T$ has at most $n$ edges if and only if $T \in B_{\ut}(n)$.
        \item Let $T \in \mathrm{LUT}^1(X)$. Then $T$ has at most $n$ edges if and only if $T \in B_{\mathrm{LUT}}(n)$. 
    \end{enumerate}
\end{lemma}

\begin{proof}
    We show (i), with (ii) following similarly. Let $T \in \ut^1(X)$. 
    
    First note that if $T$ has $n$ edges, then $T^+$ and $T^*$ also have $n$ edges by Theorem~\ref{thm:ut}. Similarly, if $S$ and $T$ has $n$ and $m$ edges respectively, then $ST$ has $n+m$ edges. The result then follows since each generator $x \in X$ corresponds to a tree with exactly one edge.
\end{proof}

\begin{proposition} \label{prop:ballisedges}
    Let $n \in \N_0$.
    \begin{enumerate}[(i)]
        \item Let $T \in \fad(X)$. Then $T$ has at most $n$ edges if and only if $T \in B(n)$.
        \item Let $T \in \flad(X)$. Then $T$ has at most $n$ edges if and only if $T \in B_L(n)$.
    \end{enumerate}
\end{proposition}
\begin{proof}
    Again, we only show (i) with (ii) following similarly. Let $T \in \fad(X)$.

    Let $\nu \colon \fad(X) \to \ut^1(X)$ be the natural injective map that maps a tree to the same tree embedded in $\ut^1(X)$ (note that this is not a morphism of unary monoids).

    Suppose $T$ has at most $n$ edges. Then certainly $\nu(T)$ has at most $n$ edges and hence $\nu(T) \in B_{\ut}(n)$. Thus, as $\fad(X)$ is a quotient of $\ut^1(X)$, $T \in B(n)$.

    Now, suppose $T \in B(n)$ and let $w \in \fbu(X)$ be a word of length at most $n$ representing $T$ in $\fad(X)$. Let $w' = \mu(w) \in \ut^1(X)$ where $\mu$ is the canonical surjective morphism $\fbu(X) \to \ut^1(X)$. Then, $w' \in B_{\ut}(n)$ and hence $w'$ is a tree with at most $n$ edges. Finally, $T = \overline{w'}$ and so $T$ has at most $n$ edges, as retracting only can remove edges.
\end{proof}

We immediately deduce the following, which allows us to compute the \textit{spherical growth} of $\flad(X)$ and $\fad(X)$ by counting edges of trees.

\begin{corollary}\label{cor:sphereisedges}Let $n \in \N_0$.
    \begin{enumerate}[(i)]
        \item Let $T \in \flad(X)$. Then $T$ has exactly $n$ edges if and only if $T \in S_L(n)$.
        \item Let $T \in \fad(X)$. Then $T$ has exactly $n$ edges if and only if $T \in S(n)$.
    \end{enumerate}
\end{corollary}

\subsection{Growth of the monogenic free left adequate monoid}

We now aim to compute $|S_L(n)|$ for the monogenic free left adequate monoid $\flad(a)$ using Corollary \ref{cor:sphereisedges}. We first consider the following set and its cardinality -- for $n \in \mathbb{N}_0$ and $k \leq n$, we define\[S_L(n,k) := \left\{T \in \mathrm{FLAd}(a)\colon  T \textrm{ has } n \textrm{ edges and }k \textrm{ trunk edges}\right\}.\]Recall we say there exists a branch at vertex $v_i$ if there exists a non-trunk edge with initial vertex $v_i$.
\begin{proposition} \label{prop:SumFormForTrees}
    For all $0 \leq k \leq n$,
    \begin{align*}
        |S_L(n,k)| &= \sum_{Y \subseteq [k]} Q\left( n-k - \sum_{i \in Y} i, |Y| \right).
    \end{align*} 
\end{proposition}
\begin{proof}
    Fix an integer $n \geq 0$, an integer $0 \leq k \leq n$. For any set $Y \subseteq [k]$, we first count the number of retract-free left $a$-trees $T$ which have $n$ edges and trunk vertices $(v_0,v_1,\dots,v_k)$ such that there exists a branch at $v_{k-i}$ if and only if $i \in Y$. Call the set of such trees $\mathcal{T}_Y$.

    Consider a tree $T \in \mathcal{T}_Y$. Recall (for example, from the proof of Lemma \ref{lem:removebranches}), that we may write $T$ as $|Y|$ idempotents of single, non-splitting paths interspersed within a product of $k$ $a$'s. For each $i \in Y$, let $\ell_i$ be the length of the non-trunk branch at vertex $v_{k-i}$. Note that $\sum_{i \in Y}\ell_i$ is the total number on non-trunk edges of $T$.
    
    Define $\phi(T)$ to be the $|Y|$-tuple $(\ell_i-i)_{i \in Y}$, where we order $Y$ from largest to smallest. 
    
    Note that $\ell_i > i$ for all $i$ -- there exists a path of $i$ trunk edges from $v_{k-i}$ to the end vertex, and so if $\ell_i \leq i$ there would exist a non-trivial retraction of the branch onto the trunk. Moreover, if $i, j \in Y$ with $i < j$, then $\ell_j > j-i+\ell_i$ as otherwise there would be a retraction of the branch at $v_{k-j}$ onto the trunk and the branch at $v_{k-i}$. Hence $\ell_j -j > \ell_i - i$. Thus the tuple $\phi(T)$ is a strictly decreasing tuple of positive integers, with total sum $\sum_{i \in Y} \ell_i - i = \sum_{i \in Y} \ell_i - \sum_{i\in Y}i = n-k-\sum_{i \in Y}i$, that is $\phi(T)$ is a partition of $n-k-\sum_{i \in Y}i$ into $|Y|$ distinct parts.

    We now claim that $\phi$ is a bijection from $\mathcal{T}_Y$ to the set of all such partitions. Indeed given such a partition $\tau = (\tau_0,\tau_1,\dots,\tau_r)$ where $Y = \{x_0 < \dots < x_r\}$, define $\psi(\tau)$ to be the tree \[(a^{n_0})^+\prod_{i=1}^{k}a(a^{n_i})^+ \textrm{ where }n_i = \begin{cases}
        i + \tau_j &\textrm{if }i = k-x_j,\\
        0 &\textrm{otherwise.}
    \end{cases}\]
    It is easily verified that $\phi$ and $\psi$ are inverse functions and thus the sets in question have equal cardinality, i.e. $|\mathcal{T}_Y| = Q(n-k-\sum_{i \in Y}i, |Y|)$. Since the subsets of $[k]$ clearly partition all retract-free left $x$-trees with $n$ edges and $k$ trunk edges via the assignment of $Y$ above, it follows that\[|S_L(n,k)| = \sum_{Y \subseteq [k]}|\mathcal{T}_Y| =\sum_{Y \subseteq [k]} Q\left( n-k - \sum_{i \in Y} i, |Y| \right).\qedhere\]
\end{proof}

To proceed, we require the following technical lemma.

\begin{lemma} \label{lem:SumsWithT}
    The following equality holds, for all $m,r,t \in \N$,
    \begin{equation} \label{eq:thelemma} \tag{$\star$}
        \sum_{Y \subseteq[r-1]} Q\left( m - \sum_{i \in Y} i, |Y| + t \right) =
    \sum_{Y \subseteq[r-2]} Q\left( m + t + 1 - \sum_{i \in Y} i, |Y| + t + 1 \right)
    \end{equation}
    where we define $[-1] = \emptyset$.
\end{lemma}
\begin{proof}
    We prove this equality by induction on $r$. 
    For the base case, let $r = 1$. 
    By considering all possible choices for $Y$ on the left hand side (LHS) of (\ref{eq:thelemma}), namely $Y = \emptyset$ or $\{0\}$, we see the LHS of (\ref{eq:thelemma}) is equal to $Q(m,t) + Q(m,t+1)$.
    
    Similarly, by considering all possible choices for $Y$ on the right hand side (RHS), namely only $Y = \emptyset$, we get the RHS of (\ref{eq:thelemma}) is equal to $Q(m + t + 1,t+1).$
    
    By using the equality $Q(n,k) = Q(n-k,k-1) + Q(n-k,k)$ from Lemma~\ref{lem:FolklorePartitions}, we see that these are equal. Thus, we have shown the base $r=1$ case.
    
    Now assume the result holds for all $l < r$ for some $r > 1$ -- we show our result for $r$. Note that, by splitting subsets in the LHS of (\ref{eq:thelemma}) into those that do and do not contain $r-1$, we see  it is equal to the following:
    \[
        \sum_{Y \subseteq[r-2]} Q\left( m + 1 - r - \sum_{i \in Y} i, |Y| + t +1\right) + \sum_{Y \subseteq[r-2]} Q\left( m - \sum_{i \in Y} i, |Y| + t \right).\]
    Similarly, splitting the subsets in the RHS of (\ref{eq:thelemma}) into those that do and do not contain $r-2$, we see that it is equal to the following:
    \begin{multline*}
        \sum_{Y \subseteq[r-3]} Q\left( m + t + 3 - r  - \sum_{i \in Y} i, |Y| + t +2\right)+ \\
        \sum_{Y \subseteq[r-3]} Q\left( m + t + 1 - \sum_{i \in Y} i, |Y| + t + 1\right) .
    \end{multline*}
    Finally, note that the first terms and the second terms of the two sums are equal by induction. Thus, we are done.
\end{proof}

Let $T \in \flad(a)$ have trunk vertices $(v_0,\dots,v_k)$. Our arguments will require counting the following set of trees -- for $n,k,l \in \mathbb{N}_0$ with $l \leq k < n$, we define
\[
S_L(n,k,l) = \left\{ T \in \flad(a)\ : \begin{array}{ll}
    &T \text{ has }n\text{ edges, }k \text{ trunk edges, and} \\
    &l = \min_{0 \leq i \leq k}\{i : \textrm{there exists a branch at }v_i\}
  \end{array}\right\}
  \]

\begin{example}
    Figure \ref{fig:sl621} shows the two $a$-trees in $S_L(6,2,1)$.

    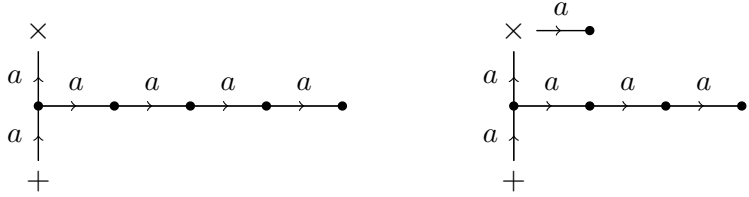
\begin{figure}[h]
        \centering
        \begin{tikzpicture}
            \setup

            \node () at ( 0,0) {};
            
            \node (0) at ( 0,0.5) {\large$+$};
            \Vertex[x=0,y=1.5]{1}
            \node (2) at ( 0,2.5) {\large$\times$};

            \Vertex[x=1,y=1.5]{L1}
            \Vertex[x=2,y=1.5]{L2}
            \Vertex[x=3,y=1.5]{L3}
            \Vertex[x=4,y=1.5]{L4}

            \Edge(0)(1)\draw (0) -- (1) node [midway, left=2pt] {$a$};
            \Edge(1)(2)\draw (1) -- (2) node [midway, left=2pt] {$a$};
            \Edge(1)(L1)\draw (1) -- (L1) node [midway, above=2pt] {$a$};
            \Edge(L1)(L2)\draw (L1) -- (L2) node [midway, above=2pt] {$a$};
            \Edge(L2)(L3)\draw (L2) -- (L3) node [midway, above=2pt] {$a$};
            \Edge(L3)(L4)\draw (L3) -- (L4) node [midway, above=2pt] {$a$};

        \end{tikzpicture}
        \hspace{0.1\textwidth}
        \begin{tikzpicture}
            \setup

            \node () at ( 0,0) {};
            
            \node (0) at ( 0,0.5) {\large$+$};
            \Vertex[x=0,y=1.5]{1}
            \node (2) at ( 0,2.5) {\large$\times$};

            \Vertex[x=1,y=1.5]{L1}
            \Vertex[x=2,y=1.5]{L2}
            \Vertex[x=3,y=1.5]{L3}
            
            \Vertex[x=1,y=2.5]{U1}
            
            \Edge(0)(1)\draw (0) -- (1) node [midway, left=2pt] {$a$};
            \Edge(1)(2)\draw (1) -- (2) node [midway, left=2pt] {$a$};
            \Edge(1)(L1)\draw (1) -- (L1) node [midway, above=2pt] {$a$};
            \Edge(L1)(L2)\draw (L1) -- (L2) node [midway, above=2pt] {$a$};
            \Edge(L2)(L3)\draw (L2) -- (L3) node [midway, above=2pt] {$a$};
            \Edge(2)(U1)\draw (2) -- (U1) node [midway, above=2pt] {$a$};

        \end{tikzpicture}
        \caption{The trees of $S_L(6,2,1)$.}
        \label{fig:sl621}
    \end{figure} 
    
\end{example}

\begin{lemma}\label{lem:BnMinusrMinus1}
    For $0 \leq l \leq k < n$, $|S_L(n,k,l)| = |S_L(n-k-1,k-l)|$.
\end{lemma}

\begin{proof}
    We perform induction on $n + k + l$. As $l \leq k < n$, our base case is given by $n + k + l = 1$, occurring exactly when $n = 1$ and $k = l = 0$. In particular note that $|S_L(1,0,0)| = 1$ and $|S_L(1-0-1,0-0)| = |S_L(0,0)| = 1$ and so our base case holds.

    Now suppose that there exists some positive integer $N$ such that for all integers $0 \leq c \leq b < a$ with $a + b + c < N$, we have $|S_L(a,b,c)| = |S_L(a-b-1,b-c)|$. We show our result for integers $0 \leq l \leq k < n$ with $n + k + l = N$.

    First, suppose $l \neq 0$. One may obtain a bijection $S_L(n,k,l) \to S_L(n-l,k-l,0)$ by erasing the first $l$ trunk edges of a given tree and moving the start vertex to the $l^{\text{th}}$ trunk vertex. Thus $|S_L(n,k,l)| = |S_L(n-l,k-l,0)|$. Certainly we have $(n-l) + (k-l) + 0 = n + k -2l < N$ and thus we apply induction to see that $|S_L(n-l,k-l,0)| = |S_L((n-l)-(k-l)-1,(k-l)-0)| = |S_L(n-k-1,k-l)|$, and the result follows.

    Now instead suppose $l = 0$ -- we aim to show that $|S_L(n,k,0)| = |S_L(n-k-1,k)|$. By Proposition \ref{prop:SumFormForTrees}, note that
    \[|S_L(n-k-1,k)| = \sum_{Y \subseteq [k]} Q\left( n -2k-1- \sum_{i \in Y} i, |Y| \right)\]\[|S_L(n,k,0)| = \sum_{\substack{Y \subseteq [k] \\ k \in Y}} Q\left( n -k- \sum_{i \in Y} i, |Y| \right) = \sum_{Y \subseteq [k-1]} Q\left( n - 2k- \sum_{i \in Y} i, |Y| + 1 \right)\]By applying Lemma \ref{lem:SumsWithT} with $m = n-2k-1$, $r = k+1$ and $t = 0$, we exactly have that $|S_L(n-k-1,k)| = |S_L(n,k,0)|$ as required.    
\end{proof}

\begin{theorem}\label{thm:treeswithktrunk}
    For $0 \leq k \leq n$, $|S_L(n,k)| = P(n+1,k+1)$.
\end{theorem}

\begin{proof}
    We begin by noting that certainly $|S_L(n,0)| = 1 = P(n+1,1)$ and certainly $|S_L(n,n)| = 1 = P(n+1,n+1)$. Thus we may suppose $1 \leq k < n$.

    We proceed by induction on $n + k$. Note that our base case is given by $k=1$ and $n=2$, in which case certainly $|S_L(n,k)| = 1 = P(n+1,k+1)$.

    Now let $n,k$ be given such that $1 \leq k < n$ and $n+k \geq 4$, and suppose that $|S_L(a,b)| = P(a+1,b+1)$ for all $a + b < n + k$. Note that \[P(n+1,k+1) = \sum_{i=1}^{k+1} P(n-k,i) = \sum_{l=0}^{k} |S_L(n-k-1,l)| = \sum_{l=0}^{k} |S_L(n-k-1,k-l)|\]where the first equality follows from repeated applications of Lemma \ref{lem:FolklorePartitions}, the second from inductive assumption, and the third from summation manipulation. Hence by Lemma \ref{lem:BnMinusrMinus1}, we have \[P(n+1,k+1) = \sum_{l=0}^{k} |S_L(n,k,l)| = |S_L(n,k)|.\qedhere\]
\end{proof}

\begin{corollary}
    For all $n \in \N_0$, $|S_L(n)| = P(n+1)$ in $\flad(a)$.
\end{corollary}
\begin{proof}
    By Corollary \ref{cor:sphereisedges} and Theorem \ref{thm:treeswithktrunk}, we observe that
    \[ |S_L(n)| = \sum_{k=0}^{n} |S_L(n,k)| = \sum_{k=1}^{n+1} P(n+1,k) = P(n+1). \qedhere\]
\end{proof}

\begin{corollary}
    The monogenic free left adequate monoid has intermediate growth.
\end{corollary}

\begin{proof}
    By a famous result of Hardy and Ramanujan \cite{HAR}, \[P(n) \sim \frac{1}{4n\sqrt{3}} \exp\left(\pi\sqrt{\frac{2n}{3}}\right).\qedhere \]
\end{proof}

\subsection{Growth of the monogenic free adequate monoid}
We now study the growth rate of $\fad(a)$, appealing to Corollary \ref{cor:sphereisedges}. It is clear that the number of trees in $\fad(a)$ grows much faster than in the one-sided case -- in fact, we will show that even the semilattice of idempotents grows exponentially. 

For $n \in \N_0$, recall that the sphere of radius $n$ in $\fad(a)$ is denoted $S(n)$. We let $S_E(n) := S(n) \cap E(\fad(a))$, that is the set of idempotents of $\fad(a)$ with exactly $n$ edges. By definition, $|S_E(n)| < |S(n)|$.

We say $T \in \ut^1(a)$ is a \emph{zig-zag tree} if the undirected graph of $T$ is a single (non-branching) path with the start and end vertex being located at a common end. We say a zig-zag tree has \emph{height} $i$ if $T$ has exactly $i$ edges pointing away from the start vertex.

\begin{example}
    Figure \ref{fig:zigzags} shows three zig-zag trees of heights $3$, $3$, and $2$ respectively (left to right). Only the rightmost tree is retract-free.

    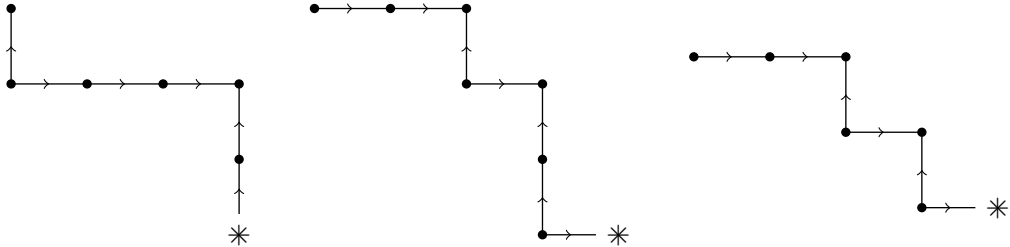
\begin{figure}[h!]
    \centering
    {\begin{tikzpicture}
        \setup
        
        \node (A) at ( 0,0) {\large$\PlusCross$};
        \Vertex[x=0,y=1]{1};
        \Vertex[x=0,y=2]{2};
        \Vertex[x=-1,y=2]{3};
        \Vertex[x=-2,y=2]{4};
        \Vertex[x=-3,y=2]{5};
        \Vertex[x=-3,y=3]{6};

        \Edge(A)(1);
        \Edge(1)(2);
        \Edge(3)(2);
        \Edge(4)(3);
        \Edge(5)(4);
        \Edge(5)(6);
        
    \end{tikzpicture}
    \hspace{1em}
    \begin{tikzpicture}
        \setup
        
        \node (A) at ( 0,0) {\large$\PlusCross$};
        \Vertex[x=-1,y=0]{1};
        \Vertex[x=-1,y=1]{2};
        \Vertex[x=-1,y=2]{3}
        \Vertex[x=-2,y=2]{4};        
        \Vertex[x=-2,y=3]{5};
        \Vertex[x=-3,y=3]{6};
        \Vertex[x=-4,y=3]{7};

        \Edge(1)(A);
        \Edge(1)(2);
        \Edge(2)(3);
        \Edge(4)(3);
        \Edge(4)(5);        
        \Edge(6)(5);
        \Edge(7)(6);
        
    \end{tikzpicture}
    \hspace{1em}
    \begin{tikzpicture}
        \setup

        \node (_) at ( 0,0) {}; 
                
        \node (A) at ( 0,0.5) {\large$\PlusCross$};
        \Vertex[x=-1,y=0.5]{1};
        \Vertex[x=-1,y=1.5]{2};
        \Vertex[x=-2,y=1.5]{3};
        \Vertex[x=-2,y=2.5]{4};
        \Vertex[x=-3,y=2.5]{5};
        \Vertex[x=-4,y=2.5]{6};

        \Edge(1)(A);
        \Edge(1)(2);
        \Edge(3)(2);
        \Edge(3)(4);
        \Edge(5)(4);
        \Edge(6)(5);
        
    \end{tikzpicture}
    }
    \caption{Three zig-zag trees with $a$-edge labels omitted.}
    \label{fig:zigzags}
    \end{figure}
        
\end{example}

Fix some integer $n$ and some integer $i < n/2$. We define the tree $p_{n,i}$ as follows: $p_{n,i}$ is the zig-zag path with $n-2i-1$ edges oriented towards the start, followed by $2i$ edges with alternating orientations beginning with an edge facing away from the start, followed by one final edge oriented towards the start. Note that $p_{n,i}$ has exactly $n$ edges and height $i$. The zig-zag tree $p_{n,3}$ is shown in Figure \ref{fig:p}.

\begin{figure}[h]
    \centering
    \begin{tikzpicture}
        \setup
        \node (A) at ( 4,0) {\large$\PlusCross$};
        \Vertex[x=0,y=0]{G0}
        \Vertex[x=1,y=0]{G1}
        \Vertex[x=2,y=0]{G2}
        \Vertex[x=3,y=0]{G3}
        \Vertex[x=0,y=1]{U1}
        \Vertex[x=-1,y=1]{U}
        \Vertex[x=-1,y=2]{L2}
        \Vertex[x=-2,y=2]{L3}
        \Vertex[x=-2,y=3]{L4}
        \Vertex[x=-3,y=3]{L5}
        \Vertex[x=-4,y=3]{L6}

        \node (L) at ( -4.2,3.2) {};
        \node (R) at ( 4.3,3.2) {};
        \node (S) at ( 4.3,-0.2) {};

        \draw [decorate, decoration={brace, amplitude=10pt}, thick]    (L) -- (R) node[midway, yshift=20pt] {$n-3$};
        \draw [decorate, decoration={brace, amplitude=10pt}, thick]    (R) -- (S) node[midway, xshift=20pt] {$3$};
        
        \Edge(G3)(A)\draw (G3) -- (A);
        \node (dots) at ( 2.5, 0) {$\cdots$};
        
        \Edge(G1)(G2)\draw (G1) -- (G2);
        \Edge(G0)(G1)\draw (G0) -- (G1);
        \Edge(G0)(U1)\draw (G0) -- (U1);
        \Edge(U)(U1)\draw (U) -- (U1);
        \Edge(U)(L2)\draw (U) -- (L2);
        \Edge(L3)(L2)\draw (L3) -- (L2);
        \Edge(L3)(L4)\draw (L3) -- (L4);
        \Edge(L5)(L4)\draw (L5) -- (L4);
        \Edge(L6)(L5)\draw (L6) -- (L5);

    \draw[-, red, dashed] (L4) to (2,3);
    \draw[-, red, dashed] (3,3) to (4,3);
    \draw[-, red, dashed] (A) to (4,3);
    \node[red] (dots) at ( 2.5, 3) {$\cdots$};

    \draw[->-, blue, dashed] (L4) to (-1,3);
    \draw[->-, blue, dashed] (-1,2) to (-1,3);
    \draw[->-, blue, dashed] (-1,2) to (1,2);
    \draw[->-, blue, dashed] (1,1) to (1,2);
    \draw[->-, blue, dashed] (1,1) to (2,1);
    \draw[->-, blue, dashed] (3,1) to (4,1);
    \draw[->-, blue, dashed] (4,0) to (4,1);
    \node[blue] (dots) at ( 2.5, 1) {$\cdots$};

    \draw[-, dashed] (-1,0) to (0,0);
    \draw[-, dashed] (-4,3) to (-1,0);
    \end{tikzpicture}
    \caption{The zig-zag tree $p_{n,3}$, shown by the solid black line.}
    \label{fig:p}
\end{figure}
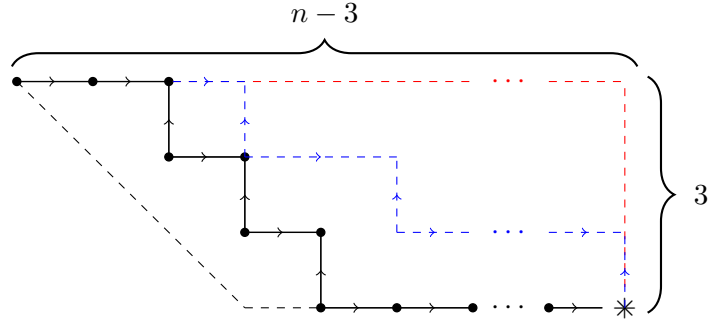 

For zig-zag trees $S,T$ with $n$ edges, we say $T \geq S$ if for all $k \leq n$, the first $k$ edges of $T$ contain at least as many edges pointing away from the start then the first $k$ edges of $S$.
For $n \in \N$ and $i < n/2$, let $Z(n,i)$ be the set of zig-zag trees $T$ with $n$ edges and height $i$ such that $T \geq p_{n,i}$.

\begin{remark}
    If we draw all edges away from the start vertex vertically and all edges towards the start vertex horizontally, as in Figure~\ref{fig:p}, then $T \in Z(n,i)$ if and only if $T$ is entirely contained between the solid black line to the red dashed line in Figure~\ref{fig:p}. For example, the blue dashed line in Figure~\ref{fig:p} is a tree in $Z(n,i)$.
\end{remark}

\begin{lemma} \label{lem:noofzigzags}
    For $n \in \N$ and $0 \leq i < n/2$, we have
    \[ |Z(n,i)| = \frac{n-2i}{n}\begin{pmatrix}
    n \\
    i
\end{pmatrix}.\]
\end{lemma}
\begin{proof}
    Note that there are $\begin{pmatrix}
        n \\ i
    \end{pmatrix}$ zig-zag trees in total with $n$ edges and height $i$.
    Note that a zig-zag tree $T$ has $T \geq p_{n,i}$ if and only if $T$ remains strictly above the diagonal dashed line in Figure \ref{fig:p} everywhere except one vertex (the vertex furthest from the start). The problem of counting such trees is exactly the \textit{Ballot Problem} \cite{AND,BER}, where the ``winning" candidate receives $n-i$ votes and the ``losing" candidate receives $i$ votes. The problem is widely known to have the solution\[\frac{n-2i}{n}\begin{pmatrix}
    n \\
    i
\end{pmatrix}.\qedhere\]
\end{proof}

\begin{proposition} \label{prop:CountingZigZags}
For all $n \in \N$, and $0 \leq k < n/2$. Then,
    \[ \sum_{i=0}^k|Z(n,i)| = \begin{pmatrix}
    n-1 \\
    k
\end{pmatrix}.\]
\end{proposition}
\begin{proof}
    Note that, $\frac{n-2i}{n}\binom{n}{i} = \binom{n}{i} - 2\binom{n-1}{i-1}$. Then, as $\binom{n}{i} = \binom{n-1}{i} + \binom{n-1}{i-1}$, we obtain that
    \[ \sum_{i=0}^k|Z(n,i)| = \sum_{i=0}^k \binom{n}{i} -2\binom{n-1}{i-1}  = \sum_{i=0}^k \binom{n-1}{i} -\binom{n-1}{i-1}  = \binom{n-1}{k}. \qedhere \]
\end{proof}

\begin{lemma}\label{lem:dvaluepreserved}
    For a zig-zag tree $T$ and $v$ a vertex of $T$, let $d(v)$ be the number of edges between the start vertex and $v$ which are oriented towards the start, minus the number of edges between the start vertex and $v$ which are oriented away from the start. Then, $d(\psi(v)) = d(v)$ for any retraction $\psi\colon T \to T$.
\end{lemma}
\begin{proof}
    Let $T$ be a zig-zag tree with $n$ edges. The only tree with $n = 0$ edges is the trivial tree and thus our statement holds. Now suppose, for induction, that all zig-zag trees of length $k < n$ for some $n \in \N$ fulfil our statement.
    
    Let $T$ be a zig-zag tree with $n$ edges and let $\psi \colon T \to T$ be a retraction. Let $S$ be the length $n-1$ subtree containing the start vertex.
    
    As $S$ is a subtree of $T$ containing the start vertex, $\psi$ restricts to a retraction of $S$, and hence, by induction, $d(\psi(v)) = d(v)$ for all vertices $v$ of $S$. It only remains to show the result for the unique vertex in $T$ and not $S$. Let $u$ be the vertex of $T$ not contained in $S$, $v$ be the final vertex of $S$, and $e$ be the unique edge between them. We show the result if $e$ is oriented towards $v$, with the other orientation dual.
    
    If $e$ is from $u$ to $v$, then $\psi(u)$ either has one more edge oriented towards the start than $\psi(v)$, or has one less edge oriented away from the start. In either case, 
    \[d(u) = d(v) + 1 = d(\psi(v)) + 1 = d(\psi(u))\]and the result holds.
\end{proof}

\begin{lemma} \label{lem:ZigZagRetractFree}
    Let $T \in Z(n,i)$. Then, $T$ is retract-free.
\end{lemma}

\begin{proof}
    Suppose $T \in Z(n,i)$ and let $\ell$ be the vertex furthest from the start vertex. Suppose $\psi : T \to T$ is a retraction -- we show that $\psi$ is the trivial map.
    
    By Lemma \ref{lem:dvaluepreserved}, $d(\psi(\ell)) = d(\ell) = n-2i$. Vertices with this $d$-value are those which appear on the diagonal dashed line in Figure \ref{fig:p}. Note however that since $T \geq p_{n,i}$, $\ell$ is the unique vertex of $T$ with $d$-value $n-2i$, and thus $\ell = \psi(\ell)$. As the start vertex is also fixed by $\psi$, it follows that all vertices are fixed by $\psi$.
\end{proof}

\begin{theorem}\label{thm:fadgrowth}
    The idempotent growth rate of $\fad(a)$ is exponential of degree at least $2$.
\end{theorem}
\begin{proof}
    By Proposition~\ref{prop:CountingZigZags}, Lemma~\ref{lem:ZigZagRetractFree}, and Stirling's approximation, we obtain the following.
\[ S_E(n) \geq\sum_{i=0}^{\floor{\frac{n-1}{2}}} |Z(n,i)| = \begin{pmatrix} n-1 \\ \floor{\frac{n-1}{2}}\end{pmatrix} \sim \frac{2^{n-1}}{\sqrt{(n-1)\pi}}.\] 
Thus, the idempotent growth rate of $\fad(a)$ is exponential of degree at least $2$.
\end{proof}

This lower bound poses two questions: first, of optimality, and secondly, of whether counting idempotents is sufficient to study the growth of $\fad(a)$. Via explicit calculation, one may verify the results in Table \ref{tab:sizeoffad}.


\begin{table}[h!]
\bgroup
\def\arraystretch{1.2} 
\begin{tabular}{r|cccccc}
$n$        & 0 & 1 & 2 & 3  & 4  & 5  \\ \hline
$|S_E(n)|$ & 1 & 2 & 3 & 6  & 11 & 28 \\
$|S(n)|$   & 1 & 3 & 6 & 14 & 29 & 74
\end{tabular}
\vspace{0.5em} 
\caption{Size of spheres in $\fad(a)$ for $0 \leq n \leq 5$.}
\label{tab:sizeoffad}
\egroup
\end{table}

\vspace{-2em} 

\begin{question}
    What is the idempotent growth rate of $\fad(a)$?
\end{question}

\begin{question}
    Is the growth rate of $\fad(a)$ of the same exponential degree as the idempotent growth rate?
\end{question}

\subsection{Higher rank}

Note that all of our results here rely heavily on a singleton generating set. In the higher rank, we obtain different results.

\begin{proposition}
    The free left adequate monoid $\flad(X)$ of rank at least $|X|\geq 2$ has exponential growth.
\end{proposition}
\begin{proof}
    Note that $X \subseteq \flad(X)$ generates, as a monoid, a free monoid of rank $|X|$ and hence grows exponentially when $|X| \geq 2$.
\end{proof}

In the two-sided case, as $\fad(a)$ embeds in the free adequate monoid of any rank at least $1$, we also deduce the following from Theorem \ref{thm:fadgrowth}.

\begin{corollary}
   The free adequate monoid of rank at least $1$ has exponential growth.
\end{corollary} 

\begin{question}
    What is the (idempotent) growth rate of $\flad(X)$ and $\fad(X)$ for higher rank?
\end{question}

We remark on a rudimentary way to bound the growth of $\fad(X)$ in higher rank from below. Suppose $|X| \geq 2$. Consider any zig-zag tree in $T \in Z(n,i)$. Since $T$ is retract-free, note that any assignment of edge labels to $T$ from $X$ still yields a retract-free tree. Since there are $|X|^n$ such assignments, it follows from Theorem \ref{thm:fadgrowth} that the number of idempotents in $\fad(X)$ with $n$ edges is at least \[|X|^n\begin{pmatrix} n-1 \\ \lfloor\frac{n-1}{2}\rfloor\end{pmatrix} \sim \frac{(2|X|)^n}{2\sqrt{(n-1)\pi}}\]and thus the growth rate of $E(\fad(X))$ is at least $2|X|$.

\section{Identities of free adequate monoids}\label{sec:identities}

In this section, we aim to classify the identities satisfied by the monogenic free left adequate monoid. 
Recall that we view left adequate monoids in the $(2,1,0)$ signature, so in this section, we investigate the (unary monoid) identities satisfied by $(\flad(a),{}^+)$. Then, as a corollary, we give the (non-unary) \emph{monoid identities} satisfied by their monoid reducts.

\subsection{Enriched Identities}
For $u,v \in \fu(\Sigma)$, we say $u \approx v$ is an \emph{identity} and say a unary monoid $\cM$ satisfies $u \approx v$ if $\phi(u) = \phi(v)$ for all unary monoid morphisms $\phi\colon \fu(\Sigma) \to \cM$.
We say a word $u \in \fu(\Sigma)$ is \emph{non-nested} if $u$ is an element of the free monoid over $\osigma = \Sigma \cup \fmp$ where $\fmp :=\{u^+ \colon u \in \fm(\Sigma)\}$, that is, $u$ does not contain $(st^+r)^+$ for any $s,t,r \in \fu(\Sigma)$.

It follows from Lemma \ref{lem:fladnonnested} that, given any identity $u \approx v$, there exists a non-nested identity $u' \approx v'$ such that $u \approx v$ is satisfied by $\flad(a)$ if and only if $u' \approx v'$ is satisfied by $\flad(a)$. Thus, in the following, we give our results for non-nested identities.

For $u \in \fm(\ol{\Sigma})$ and $x \in \osigma$, let $|u|_x$ be the number of times $x$ occurs in $u$. Note that, especially in the monogenic case, $|u|_a$ may be different from the previously used $|u|$. In particular, occurrences of $a^+$ in $u$ do not count towards $|u|_a$, that is, $|a^+|_a = 0$.

Let $\supp(u)$ be the set containing $x \in \osigma$ if $|u|_x \geq 1$. Let $\maxp_u(x)$ be the longest prefix of $u$ ending in $x$. Let $\maxp_u'(x)$ be $\maxp_u(x)$ with the $x$ removed from the right. Let $\ol{u} \in \fm(\Sigma)$ be the word obtained by removing every element of $(\fm(\Sigma))^+$ from $u$.
\begin{example}
    Let $u = ab^+bb^+aab^+(ab)^+(ab)^+$. Then \begin{itemize}
        \item $\supp(u) = \{a,b,b^+,(ab)^+\}$,
        \item $|u|_a = 3$ and $|u|_b = 1$,
        \item $\maxp_u(b^+) = ab^+bb^+aab^+$ and $\maxp'_u(a) = ab^+bb^+aa$.
        \item $\overline{u} = abaa$.
    \end{itemize}
\end{example}
For $u \in \fm(\ol{\Sigma})$ and $x \in \supp(u)$, let $\suff_u(x)$ be
\begin{enumerate}[(i)]
    \item the shortest suffix of $u$ containing $x$ if $x \in \Sigma$,
    \item the longest suffix of $u$ with no element from $\Sigma$ to the left of any copy of $x$ if $x \in \fm(\Sigma)^+$.
\end{enumerate}

In other words, if $x \in \fm(\Sigma)^+$ and $v$ is the shortest suffix of $u$ containing $x$ then, $\suff_u(x)$ is the longest suffix of $u$ with $|\suff_u(x)|_y = |v|_y$ for all $y \in \Sigma$, that is, any elements from $\fm(\Sigma)^+$ to the left of $x$ in $u$ are also included but no more elements from $\Sigma$.

\begin{example}
    Let $u = (ab)^+ab^+(ab)^+bc(ba)^+baa$. Then \begin{itemize}
        \item $\suff_u(b) = baa$,
        \item $\suff_u((ab)^+) = b^+(ab)^+bc(ba)^+baa$.
    \end{itemize}
\end{example}

Given a morphism $\phi \colon \fu(\Sigma) \to \flad(a)$, let $\rho_\phi \colon \fu(\Sigma) \to \N$ be the morphism obtained by mapping each $y \in \flad(a)$ to the trunk length of $\phi(y)$. 

\begin{lemma} \label{lem:TrunksPref}
    Let $w \in \fm(\Sigma)$ and $\phi \colon \fm(\Sigma) \to \flad(a)$. Then,
    \[ |\phi(w)^+| = \max_{x \in \supp(w)} (\rho_\phi(\maxp_w'(x)) + |\phi(x)^+|)\]
\end{lemma}
\begin{proof}
    Note that $|\phi(w)^+|$ is the length of any longest path beginning from the start in $\phi(w)$. 
    Such a path must end at an edge in $\phi(x)$ for some $x \in \supp(w)$.
    It follows from Lemma~\ref{lem:removebranches} that $\phi(w)^+ = (\Theta(\maxp'_w(x))x)^+$ and hence \[|\phi(w)^+| = |(\Theta(\maxp'_w(x))\phi(x))^+| = \rho_\phi(\maxp_w'(x) + |\phi(x)^+|\]For any letter $y \in \supp(w)$ with $w \neq x$, all paths in $\phi(w)$ which end in $\phi(y)$ must be at most as long as the path ending in $\phi(x)$. It follows that \[|\phi(w)^+| = \max_{x \in \supp(w)} (\rho_\phi(\maxp_w'(x)) + |\phi(x)^+|)\]as required.
\end{proof}

For $u \in \fm(\osigma)$, $w^+ \in \supp(u)$, and $x \in \supp(w)$, define $P^u_{w,x}, Q^u_{w,x}, R^u_{w,x}  \subseteq \fm(\osigma)$ by
\begin{align*}
P^u_{w,x} &= \{\ol{\maxp_{\suff_u(w^+)}(k^+)}\cdot k^+ \colon k^+ \in \supp(\suff_u(w^+))\text{ and } x \in \supp(k) \} \\
Q^u_{w,x} &= \begin{cases}
    P^u_{w,x} &\text{if } x \notin \supp(\suff_u(w^+)) \\
    P^u_{w,x} \cup  \{\ol{\suff_u(w^+)} \cdot \varepsilon^+\} &\text{otherwise}.
\end{cases}\\
R^u_{w,x} &= Q^u_{w,x} \setminus \{ k^+ \colon |\maxp_w(x)|_y = |\maxp_k(x)|_y \text{ for all } y \in \Sigma\}
\end{align*}

\begin{example}
    Let $u = a(ab)^+cb^+(bc)^+bab^+ba^+c^+b^+ba$. Let $w = bc$. We have that $\suff_u(w^+) = b^+(bc)^+bab^+ba^+c^+b^+ba$ and hence\begin{itemize}
        \item $P^u_{w,b} = \{b^+,(bc)^+,bab^+,babb^+\}$,
        \item $Q^u_{w,b} = \{b^+,(bc)^+,bab^+,babb^+,babba\varepsilon^+\}$,
        \item $R^u_{w,b} = \{bab^+,babb^+,babba\varepsilon^+\}$,
        \item $P^u_{w,c} = Q^u_{w,c} = \{(bc)^+,babc^+\}$,
        \item $R^u_{w,c} = \{babc^+\}$.
    \end{itemize}
\end{example}

We will require the following construction, which we will repeatedly use in proving a classification for the identities satisfied by $(\flad(a),{}^+)$.

\begin{lemma} \label{lem:ScaleMorphismUp}
    Let $\Sigma$ be finite, $\psi \colon \fm(\Sigma) \to \N$ be a morphism, and $L \in \N$. Then, there exists a morphism $\zeta \colon \fm(\Sigma) \to \N$ such that, for all $u,v \in \fm(\Sigma)$ with $|u|,|v| < L$,
    \begin{enumerate}[(i)]
        \item if $\psi(u) > \psi(v)$ then $\zeta(u) > \zeta(v)$,
        \item if $\zeta(u) \geq \zeta(v)$ then $\psi(u) \geq \psi(v)$, and
        \item $\zeta(u) = \zeta(v)$ if and only if $|u|_y = |v|_y$ for all $y \in \Sigma$.
    \end{enumerate}
\end{lemma}
\begin{proof}
    Index $\Sigma$ by $[r-1]$, so that $\Sigma = \{\sigma_i \colon i \in [r-1]\}$, and define $\zeta \colon \fm(\Sigma) \to \N$ by $\zeta(\sigma_i) = L^r\cdot \psi(\sigma_i) + L^i$. Then, for $u \in \fm(\Sigma)$ with $|u| < L$, we have that 
    \[ L^r\cdot \psi(u) + L^r > \zeta(u) \geq L^r\cdot \psi(u).\]

    Thus, for $u,v \in \fm(\Sigma)$ with $|u|,|v| < L$, we obtain that:
    \begin{enumerate}[(i)]
        \item If $\psi(u) > \psi(v)$, then 
        \[ \zeta(u) \geq L^r\cdot \psi(u) \geq L^r \cdot \psi(v) + L^r > \zeta(v).\]
        \item If $\zeta(u) \geq \zeta(v)$, then 
        \[ \psi(u) > \frac{\zeta(u)}{L^r} -1 \geq \frac{\zeta(v)}{L^r} -1 \geq \psi(v) -1.\]
        Hence $\psi(u) \geq \psi(v)$ as $\psi(u),\psi(v) \in \N$.
        \item Note that the right-to-left implication is clear, so suppose $\zeta(u) = \zeta(v)$. Then:
    \[ L^r\cdot \psi(u) + \sum_{i \in [r-1]}L^i\cdot|u|_{\sigma_i}  =  \zeta(u) = \zeta(v) = L^r\cdot \psi(v) + \sum_{i \in [r-1]}L^i\cdot|v|_{\sigma_i}. \]
    Thus, $|u|_{\sigma_i} = |v|_{\sigma_i}$ for all $i \in [r-1]$ as $L > |u|, |v|$. \qedhere 
    \end{enumerate}
\end{proof}

For the forthcoming result, we introduce the following terminology. We say that a directed path $B$ in a tree $T \in \ut^1(X)$ \textit{strongly retracts} if for every retraction $\rho$ of $T$ with image $\overline{T}$ and every edge $e$ of $B$, we have $\rho(e) \neq e$.

Equivalently, a path strongly retracts if there is no subgraph of $T$ isomorphic to $\overline{T}$ containing any edge of $B$.

\begin{example}
    Figure \ref{fig:strongly} shows three examples of left $a$-trees with paths denoted $B_1$, $B_2$, and $B_3$ respectively. The branch $B_1$ may not retract at all. The branch $B_2$ does not strongly retract (even though there is some retraction which `erases' it). The branch $B_3$ strongly retracts.

    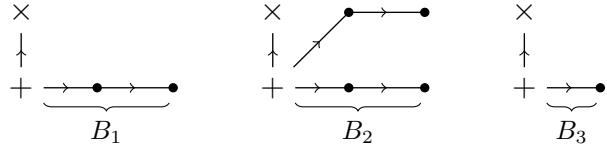
\begin{figure}[h]
        \centering
        \begin{tikzpicture}
            \setup

            \node (A) at ( 0,0) {\large$+$};
            \Vertex[x=1,y=0]{L1}
            \Vertex[x=2,y=0]{L2}
            \node (R1) at ( 0,1) {\large$\times$};
                      
            \Edge(A)(R1);
            \Edge(A)(L1);
            \Edge(L1)(L2);

            \draw [decorate,decoration={brace,raise=2mm,amplitude=3pt,mirror}] (A) -- (L2) node [font=\small, below, pos=.5, yshift=-3mm] {$B_1$};
            
        \end{tikzpicture}
        \hspace{0.05\textwidth}
        \begin{tikzpicture}
            \setup

            \node (A) at ( 0,0) {\large$+$};
            \Vertex[x=1,y=0]{L1}
            \Vertex[x=2,y=0]{L2}
            \node (R1) at ( 0,1) {\large$\times$};

            \Vertex[x=1,y=1]{L3};
            \Vertex[x=2,y=1]{L4};
                      
            \Edge(A)(R1);
            \Edge(A)(L1);
            \Edge(L1)(L2);
            \Edge(A)(L3);
            \Edge(L3)(L4);

            \draw [decorate,decoration={brace,raise=2mm,amplitude=3pt,mirror}] (A) -- (L2) node [font=\small, below, pos=.5, yshift=-3mm] {$B_2$};
            
        \end{tikzpicture}
        \hspace{0.05\textwidth}
        \begin{tikzpicture}
            \setup

            \node (A) at ( 0,0) {\large$+$};
            \Vertex[x=1,y=0]{L1}
            \node (R1) at ( 0,1) {\large$\times$};
                      
            \Edge(A)(R1);
            \Edge(A)(L1);

            \draw [decorate,decoration={brace,raise=2mm,amplitude=3pt,mirror}] (A) -- (L1) node [font=\small, below, pos=.5, yshift=-3mm] {$B_3$};
            
        \end{tikzpicture}
        \caption{Three left $a$-trees with branches denoted.}
        \label{fig:strongly}
    \end{figure}

\end{example}

We are now in a position to classify the enriched identities satisfied by $\flad(a)$.

\begin{theorem} \label{thm:fladunaryidentities}
    Let $u \approx v$ be a non-nested identity over $\Sigma$. 
    Then, $u \approx v$ is satisfied by $(\flad(a),{}^+)$ if and only if
    \begin{enumerate}[(i)]
        \item $|u|_x = |v|_x$ for all $x \in \Sigma$.
        \item For all $x \in \supp(u) \cap \Sigma$, either 
        \begin{enumerate}[(a)]
            \item there exists $w^+ \in \supp(\suff_u(x))$ with $x \in \supp(w)$, or 
            \item $|\suff_{u}(x)|_y = |\suff_{v}(x)|_y$ for all $y \in \Sigma$.
        \end{enumerate}
        \item For all $w^+ \in \mathrm{supp}(u)$ and $x \in \supp(w)$, either
        \begin{enumerate}[(a)]
            \item for every morphism $\psi\colon \fm(\Sigma) \to \N$ there exists $pk^+ \in R^u_{w,x}$ such that $\psi(\maxp_w(x)) \leq \psi(\maxp_{pk}(x))$, or
            \item there exists some $h^+ \in \supp(v)$ with $|\suff_u(w^+)|_y = |\suff_v(h^+)|_y$ and $|\maxp_w(x)|_y = |\maxp_h(x)|_y$ for all $y \in \Sigma$. 
        \end{enumerate}
            \item Dual of (iii) for $v$.
    \end{enumerate}
\end{theorem}
\begin{proof}
    Note that since $\fu(\Sigma)$ is freely generated by $\Sigma$, we may define morphisms $\phi \colon \fu(\Sigma) \to \flad(a)$ by specifying their image on $\Sigma$. We begin by showing that the conditions (i)--(iv) are necessary. Suppose $u \approx v$ is satisfied by $\flad(a)$, i.e. all morphisms $\phi \colon \fu(\Sigma) \to \flad(a)$ agree on $u$ and $v$. Throughout, we assume $\Sigma$ is finite (if not, then restrict to only letters appearing in $u$ or $v$).
    \begin{enumerate}[(i)]
        \item Let $x \in \Sigma$ and define the morphism $\phi\colon \fu(\Sigma) \to \flad(a)$ by $\phi(x) = a$ and $\phi(y) = \varepsilon$ for all $y \in \Sigma \setminus \{x\}$. Then, $|u|_x = \rho_\phi(u) = \rho_\phi(v) = |v|_x$. As $x \in \Sigma$ was arbitrary, (i) holds.

        \item Suppose (iia) does not hold for some $x \in \supp(u) \cap \Sigma$ -- we show (iib) holds. Note that $x \in \supp(v)$ by (i), so certainly $|\suff_u(x)|_x = |\suff_v(x)|_x = 1$. 
        Now, let $y \in \Sigma \setminus \{ x\}$ and define the morphism $\phi \colon \fu(\Sigma) \to \flad(a)$ by $\phi(y) = a$, $\phi(x) = a^N(a^{2N})^+a^N$ for $N > |u| + |v|$, and $\phi(z) = \varepsilon$ for all $z \in \Sigma \setminus \{x,y\}$. Then,
        \[ \phi(u) = Sa^N(a^{2N})^+a^NT \text{ and } \phi(v) = S'a^N(a^{2N})^+a^NT'\]
    for some $S,S',T,T' \in \flad(a)$ where $T$ and $T'$ have $|\suff_u(x)|_y$ and $|\suff_v(x)|_y$ trunk edges respectively.

    \par Note $|T| < |u|$ as there does not exist $w^+ \in \supp(\suff_u(x))$ with $x \in \supp(w)$ and $|\suff_u(x)|_x = 1$.
    Thus, the rightmost $(a^{2N})^+$ in $\phi(u)$ corresponds to a retract-free branch as $2N > N + |u|$. 
    Then, as $\phi(u) = \phi(v)$, we have that $\phi(v)$ has a branch of length $2N$ on the trunk vertex $|\suff_u(x)|_y+N$ edges from the end vertex.
    
    However, $\phi(v)$ contains at most one branch between $|\suff_v(x)|_y + 1$ and $|\suff_v(x)|_y + 2N -1$ edges from the end vertex, exactly $|\suff_v(x)|_y + N$ edges from the end vertex. Thus, as $N > |v| > |\suff_v(x)|_y$, we can only obtain equality if $|\suff_u(x)|_y = |\suff_v(x)|_y$.
    As the choice of $x$ and $y$ was arbitrary, (iib) holds.    

    \item Suppose (iiia) does not hold for some $w^+ \in \supp(u)$ and $x \in \supp(w)$. Then, there exists a morphism $\psi \colon \fm(\Sigma) \to \N$ such that $\psi(\maxp_w(x)) > \psi(\maxp_{pk}(x))$ for all $pk^+ \in R^u_{w,x}$. 
    Let $\zeta \colon \fm(\Sigma) \to \N$ be the morphism obtained by Lemma~\ref{lem:ScaleMorphismUp} with $\psi$ and $L > |u|+|v|$. Then, by Lemma~\ref{lem:ScaleMorphismUp}(i), $\zeta(\maxp_w(x)) > \zeta(\maxp_{pk}(x))$ for all $pk^+ \in R^u_{w,x}$.
    
    Define $\phi \colon \fu(\Sigma) \to \flad(a)$ by $\phi(x) = a^{\zeta(x)}(a^N)^+$,  $\phi(y) = a^{\zeta(y)}$ for $y \neq x$ where $N > |u| \cdot \max_{z \in \Sigma} \zeta(z)$. We now show the following lemma.

    \begin{lemma}\label{lem:isxinit}
        Let $s \in \fm(\Sigma)$ with $|s| \leq |u|$. Then \begin{itemize}
            \item If $x \notin \supp(s)$, then $|\phi(s)^+| < N$,
            \item If $x \in \supp(s)$, then $|\phi(s^+)| = \zeta(\maxp_s(x)) + N > N$.
        \end{itemize}
    \end{lemma}

    \begin{proof}
        Let $s \in \fm(\Sigma)$ with $|s| \leq |u|$. Recall Lemma \ref{lem:TrunksPref} and the map $\rho_\phi$. Note first that, for all $y \in \Sigma$, 
        \[ \rho_\phi(\maxp_s(y)) \leq \rho_\phi(s) \leq \sum_{z \in \Sigma}\rho_\phi(z)|s|_z = \sum_{z \in \Sigma}\zeta(z)|s|_z < N.\]
        By Lemma \ref{lem:TrunksPref}, there exists $y \in \Sigma$ such that $|\phi(s^+)| = \rho_\phi(\maxp_s'(y)) + |\phi(y)^+|$.
        If $y \neq x$, then our choice of $\phi$ ensures that we have $|\phi(y)^+| = \zeta(y) = |\phi(y)|$, and thus $|\phi(s^+)| = \rho_\phi(\maxp_s'(y)) + |\phi(y)| = \rho_\phi(\maxp_s(y)) < N$. Hence, if $x \notin \supp(s)$ we certainly have $|\phi(s^+)| < N$. 
        
        Conversely suppose $x \in \supp(s)$. Then, $|\phi(s^+)| \geq |\phi(x)| \geq N$, and thus we must have $y = x$. It follows then that
        \[ |\phi(s^+)| = \rho_\phi(\maxp_s'(x)) + |\phi(x)^+| = \zeta(\maxp_s'(x)) + \zeta(x) +  N = \zeta(\maxp_s(x)) + N.\]
        Since $x \in \supp(s)$, $\maxp_s(x) \neq \varepsilon$, and thus $\zeta(\maxp_s(x)) \neq 0$ by Lemma \ref{lem:ScaleMorphismUp}(iii). Hence $|\phi(s^+)| > N$.
    \end{proof} 
    
    By applying Lemma \ref{lem:isxinit} to $w$ and $\maxp_w(x)$, we see that \[|\phi(w^+)| = |\phi(\maxp_w(x)^+)| = \zeta(\maxp_w(x)) + N.\] Additionally note that the path in $\phi(u)$ contributed by $\phi(w^+)$ has initial vertex $\zeta\left(\overline{\suff_u(w^+)}\right)$ trunk edges from the end vertex. We now claim the following.
    
    \textbf{Claim:} The path contributed by $\phi(w^+)$ in $\phi(u)$ is a branch which does not strongly retract.
    
    \textit{Proof of Claim.} Certainly the longest path in $\phi(\suff_u(w^+))$ has length $|\phi(pk^+)^+|$ for some $k^+ \in \supp(\suff_u(w^+)\cdot\varepsilon^+)$ and $p = \ol{\maxp_{\suff_{u}(w^+)\cdot \varepsilon^+}(k^+)}$. Note that by taking $k^+ = \varepsilon^+$ we have $p = \ol{\suff_u(w^+)}$. To see that the longest path is of this form, note that it must either end at a vertex in some $\phi(k^+)$ or at $\phi(y)$ for some $y \in \Sigma$.

    Now, let $p$ and $k^+$ be as above. By Lemma~\ref{lem:isxinit}, if $x \not\in \supp(pk)$, then $|\phi(pk^+)^+| < N \leq |\phi(w^+)|$. Similarly, if $pk^+ \in R^u_{w,x}$ then 
    \[ |\phi(w^+)| = \zeta(\maxp_w(x)) + N  >\zeta(\maxp_{pk}(x)) + N = \rho_\phi(\maxp_{pk}'(x)) + |\phi(x)^+| = |\phi(pk^+)^+|.\] 
    Finally, if $x \in \supp(pk)$ but $x \notin \supp(k)$, then 
    \[ |\phi(pk^+)^+| = |\phi(p)^+| \leq |\phi(\ol{\suff_u(w^+)}| < |\phi(w^+)|\]
    where the final inequality follows since, in this case, $\ol{\suff_u(w^+)}\varepsilon^+ \in R^u_{w,x}$.

    In all cases above, $\phi(w^+)$ does not strongly retract. It follows that if $|\phi(pk^+)| \geq |\phi(w^+)|$, then we have $pk^+ \in P^u_{w,x} \setminus R^u_{w,x}$, that is, $p = \varepsilon$ and $|\maxp_w(x)|_y = |\maxp_k(x)|_y$ for all $y \in \Sigma$. We remark that, by Lemma~\ref{lem:isxinit}, $\phi(k^+) = \phi(w^+)$. As $p = \varepsilon$, the path corresponding to $\phi(w^+)$ in $\phi(u)$ does not {\em strongly} retract, and we have shown the claim.
    
    Now, since $\phi(u) = \phi(v)$, we have that $\phi(v)$ has a branch corresponding to $\phi(w^+)$. By our choice of $N$, $\phi(v)$ can only have a branch of length \linebreak$\zeta(\maxp_w(x)) + N > N$ at the same position if there exists $h^+ \in \supp(v)$ with $\zeta(\ol{\suff_u(w^+)}) = \zeta(\ol{\suff_v(h^+)})$ and $\zeta(\maxp_w(x)) = \zeta(\maxp_h(x))$. 
    By Lemma~\ref{lem:ScaleMorphismUp}(iii), $|\suff_u(w^+)|_y = |\suff_v(h^+)|_y$ and $|\maxp_w(x)|_y = |\maxp_h(x)|_y$ for all $y \in \Sigma$, that is, if (iiib) holds.
    
    \item Similar to (iii).
    \end{enumerate}
    
    We now show that these conditions are sufficient. 
    Let $u \approx v$ be an identity satisfying the conditions (i)--(iv) and $\phi \colon \fu(\Sigma) \to \flad(a)$ be a morphism. 
    By Theorem~\ref{thm:flad}, (i) implies that $\phi(u)$ and $\phi(v)$ have the same number of trunk edges, that is $\rho_\phi(u) = \rho_\phi(v)$.
    So it suffices to show that $\phi(u)$ and $\phi(v)$ have the same branches. Throughout, we use $\rho$ to denote $\rho_\phi$ to simplify notation.
    

    For $0 \leq m \leq \rho(u) + 1$, let $\cB(m)$ be the statement that for all $0 \leq i < m$, $\phi(u)$ has a branch $i$ edges from the end node of length $\ell_i$ if and only if $\phi(v)$ has a branch $i$ edges from the end node of length $\ell_i$.
    Clearly, $\cB(0)$ is vacuously true so, for induction, assume that $\cB(m)$ holds for some $0 \leq m \leq \rho(u)$. We show that $\cB(m+1)$ holds.
    
    Clearly, if neither $\phi(u)$ nor $\phi(v)$ has a branch $m$ edges from the end vertex, we are done. First, suppose $\phi(u)$ has a branch, $B$, of length $\ell \geq 1$, $m$ edges from the end vertex.

    We proceed as follows. The branch $B$ must be `created' either by some letter $x \in \supp(u) \cap \Sigma$, or by some $w^+ \in \supp(u)$ for a $w \in \fm(\Sigma)$. Owing to retractions, the branch $B$ may indeed be created by multiple such letters or idempotents, or a combination of both types.
    
    We begin with the former case, that is, there exists $x \in \Sigma$ such that 
    \begin{enumerate}[(U1)]
        \item $\rho(\suff_u(x)) - \rho(x) \leq m \leq \rho(\suff_u(x))$,  \label{enu:notfromanxa}
        \item $\phi(x)$ has a length $\ell$ branch $m - (\rho(\suff_u(x)) - \rho(x))$ edges from the end node. \label{enu:notfromanxb}
    \end{enumerate}
    If $x$ satisfies (iib) then $|\suff_u(x)|_y = |\suff_v(x)|_y$ for all $y \in \Sigma$ and hence $\phi(v)$ has a length $\ell$ branch $m$ edges from the end node.
    Moreover, if the branch in $\phi(v)$ can retract, it may only retract into another branch at the same trunk vertex as the subtrees of $\phi(u)$ and $\phi(v)$, consisting of the final $m-1$ trunk edges and the branches leaving them are identical, and the branch $B$ does not strongly retract in $\phi(u)$. Hence, we have shown that $\phi(v)$ has a branch of the same length as $B$ located $m$ edges from the end vertex.
    
    So, instead, suppose $x$ satisfies (iia). Then, there exists $w^+ \in \supp_u(x)$ with $\rho(\suff_u(w^+)) \leq m$ and $x \in \supp(w)$. Note that $|\phi(w)^+| \geq |\phi(x)| \geq \ell$.
    Moreover, if either $|\phi(w)^+| > \ell$ or $\rho(\suff_u(w^+)) < m$, then $B$ strongly retracts, giving a contradiction.
    Therefore $|\phi(w)^+| = \ell$ and $\rho(\suff_u(w^+)) = m$, that is, $w^+$ also `creates' $B$. Thus, we have reduced to the case where there exists $w^+ \in \supp(u)$ that `creates' $B$.

    \par Now suppose that $B$ is created by some $w^+ \in \supp(u)$. Choose $w^+ \in \supp(u)$ to be a rightmost such word, i.e. choose some $w^+ \in \supp(u)$ with the property that $|\suff_u(w^+)| \leq |\suff_u(k^+)|$ for any $k^+$ satisfying $|\phi(k)^+| = \ell$ and $\rho(\suff_u(k^+)) = m$.

    By Lemma~\ref{lem:TrunksPref}, $\ell = |\phi(w^+)| = \rho(\maxp_w'(x)) + |\phi(x)^+|$
    for some $x \in \Sigma$. Define the following sets,
    \begin{align*}
        E &= \{t^+ \in \supp(u) \colon\ x \in \supp(t) \text{ and } \suff_u(w^+) = \suff_u(t^+)\}, \text{ and} \\
        W &= \{t^+ \in E \colon \rho(\maxp_t(x)) = \rho(\maxp_w(x))\}.
    \end{align*}
    Note that, if $t^+ \in W$, then $|\phi(t^+)| \geq \ell$. Since $B$ is a branch $m = \rho(\suff_u(w^+))$ edges from the end of length $\ell$, $w^+ \in W$. Moreover, as $B$ does not strongly retract, we must have $\phi(t^+) = \phi(w^+)$ and $\phi(t^+)$ `creates' $B$ for any $t^+ \in W$. Moreover $w^+ \in W$.
    
    We now claim there exists some $t^+ \in W$ such that (iiia) fails to hold with $x$. So, for a contradiction suppose: \begin{equation}\label{eq:iiiassume}\tag{$\dagger$}
        \text{Each }t^+ \in W\text{ satisfies (iiia) with }x.
    \end{equation}
    Let $t^+ \in W$ then, by (iiia), there exists $pk^+ \in R^u_{t,x}$ such that $\rho(\maxp_t(x)) \leq \rho(\maxp_{pk}(x))$ where we abuse notation by also using $\rho$ to denote its restriction to $\fm(\Sigma)$. Then, 
    \[ |\phi(t^+)| = \rho(\maxp_t'(x))+ |\phi(x)^+| \leq \rho(\maxp_{pk}'(x))+ |\phi(x)^+| \leq |\phi(pk^+)^+| \]
    where the first equality holds as $t^+ \in W$ and $\phi(t^+) = \phi(w^+)$.
    Thus, there is some path in $\phi(pk^+)$ of length at least as long as $\phi(t^+)$. Hence $\phi(t^+)\phi(pk^+) = \phi(pk^+)$.

    \textbf{Claim:} Let $t^+ \in W$ and $pk^+ \in R^{u}_{t,x}$ such that $\rho(\maxp_t(x)) \leq \rho(\maxp_{pk}(x))$, then we have $p = \varepsilon$. 

    \textit{Proof of Claim.} First suppose $\phi(p)$ is not an idempotent then, as $\phi(t^+)$ `creates' $B$ and $\phi(t^+)\phi(pk^+) = \phi(pk^+)$, the branch $B$ strongly retracts into $\phi(pk^+)$ and we have an immediate contradiction.
    Thus, $\phi(p)$ is idempotent. By Corollary \ref{cor:productofmonoidems}, we have either $\phi(pk^+) = \phi(p)$ or $\phi(pk^+) = \phi(k^+)$ -- we split into these cases.

    \begin{description}
        \item[Case 1, $\phi(pk^+) = \phi(p)$] Write $p = p_1p_2\cdots p_n$ for letters $p_i \in \Sigma$. Certainly $\phi(p) = \prod_{1 \leq i \leq n} \phi(p_i)$, and all such $\phi(p_i)$ are idempotent. By Corollary \ref{cor:productofmonoidems}, we have $\phi(p) = \phi(p_j)$ for some $1 \leq j \leq n$.
         
        As $B$ does not strongly retract and $\phi(w^+)\phi(p_j) = \phi(p_j)$, we get that  $|\phi(p_j)| = \ell$ and $\rho(\suff_u(p_j)) = m$, and hence $p_j$ satisfies \hyperref[enu:notfromanxa]{(U1)} and \hyperref[enu:notfromanxb]{(U2)}.
        Therefore, in a similar way to above, $p_j$ either satisfies (iib), giving a contradiction, or $p_j$ satisfies (iia) and there exists $k^+ \in \supp(u)$ appearing after $p_j$ with $|\phi(k^+)| = \ell$ and $\rho(\suff_u(k^+)) = m$. Since $\suff_u(w^+)$ contains $p_j$ but $\suff_u(k^+)$ does not, $|\suff_u(k^+)| < |\suff_u(w^+)|$. However, we chose $w^+$ rightmost -- we supposed $|\suff_u(w^+)| \leq |\suff_u(k^+)|$. We arrive at a contradiction.
        
        \item[Case 2, $\phi(pk^+) = \phi(k^+)$] Similarly to Case 1, we may deduce that $\phi(k^+)$ corresponds to the branch $B$. Recall that our choice of $w^+$ ensures that $|\suff_u(w^+)|_y \leq |\suff_u(k^+)|_y$ for all $y \in \Sigma$. Here though, for any letter $y \in \supp(p)$, $|\suff_u(w^+)|_y = |\suff_u(k^+)|_y + |p|_y > |\suff_u(k^+)|_y$, contradicting our choice of $w^+$. Thus, there is no $y \in \supp(p)$ and hence $p = \varepsilon$.
    \end{description}
    Therefore, $p = \varepsilon$ and we have shown the claim.

Now, let $\zeta \colon \fm(\Sigma) \to \N$ be the morphism obtained from Lemma~\ref{lem:ScaleMorphismUp} with $\rho$ and $L > |u|$, and define 
    \[W' = \{t^+ \in E \colon \zeta(\maxp_t(x)) \geq \zeta(\maxp_s(x)) \text{ for all } s^+ \in E \}.\]

    Recall that $\rho(\maxp_w(x)) \geq \rho(\maxp_s(x))$ for all $s^+ \in E$ as the length of $\phi(w)^+$ is $\ell = |\phi(w^+)| = \rho(\maxp_w(x)) + |\phi(x)^+|$. Then, for all $t^+ \in W$ and $s^+ \in E \setminus W$, $\rho(\maxp_t(x)) > \rho(\maxp_s(x))$ and hence, $\zeta(\maxp_t(x)) > \zeta(\maxp_s(x))$ by Lemma~\ref{lem:ScaleMorphismUp}(i). Thus, if $t^+ \in W'$ then $t^+ \notin E \setminus W$ and hence, $W' \subseteq W$.

Moreover, by Lemma~\ref{lem:ScaleMorphismUp}(ii), if $t^+ \in W$ and $pk^+ \in R^u_{t,x}$ are such that \linebreak $\zeta(\maxp_t(x)) \leq \zeta(\maxp_{pk}(x))$, then $\rho(\maxp_t(x)) \leq \rho(\maxp_{pk}(x))$. By the above claim, we thus have $p = \varepsilon$.

     Certainly $W' \neq \emptyset$ by the well-ordering principle. Let $t^+ \in W'$ and recall our assumption \eqref{eq:iiiassume}. As $W' \subseteq W$, $t^+$ satisfies (iiia) with $x$, and thus, there exists $k^+ \in R^u_{t,x}$ such that $\zeta(\maxp_{t}(x)) \leq \zeta(\maxp_k(x))$. 
     Then, by the definition of $W'$, we see that $k^+ \in W'$.
    Thus, $\zeta(\maxp_t(x)) = \zeta(\maxp_k(x))$ by the definition of $W'$ and hence, by Lemma~\ref{lem:ScaleMorphismUp}(iii), $|\maxp_t(x)|_y = |\maxp_k(x)|_y$ for all $y \in \Sigma$.
    However, this implies that $k^+ \notin R^u_{t,x}$, giving a contradiction to \eqref{eq:iiiassume}. Thus, there exists a $w^+ \in W$ such that $w^+$ satisfies (iiib) with $x$.

    By condition (iiib), there exists $h^+ \in \supp(v)$ with $|\suff_u(w^+)|_y = |\suff_v(h^+)|_y$ and $|\maxp_w(x)|_y = |\maxp_h(x)|_y$ for all $y \in \Sigma$. 
    Hence, $|\phi(h^+)| \geq \ell$ and, in $\phi(v)$, corresponds to a branch $m$ edges from the end node. Moreover, the branch is retract-free as $\cB(m)$ holds -- $\phi(u)$ and $\phi(v)$ are identical for the subtrees consisting of the final $m-1$ trunk edges and any branches leaving them.
    Thus, we have shown that if $\phi(u)$ has a retract-free branch, of length $\ell \geq 1$, $m$ edges from the end vertex, then $\phi(v)$ has a corresponding branch and it is at least as long.

    We now aim to use the dual of this argument to show the same for any branch in $\phi(v)$ which is $m$ trunk edges from the end vertex. Whilst (iii) has the assumed dual (iv), we require the dual of (ii) for $v$, which we show here.

    \par Let $x \in \supp(v) \cap \Sigma$ and suppose $|\suff_u(x)|_y \neq |\suff_v(x)|_y$ for some $y \in \Sigma$. We aim to show that there exists $h^+ \in \supp(\suff_v(x))$ with $x \in \supp(h)$, that is, prove the dual of (ii) for $v$.

    Clearly $x$ does not satisfy (iib), so $x$ satisfies (iia) and hence, there exists some $w^+ \in \supp(\suff_u(x))$ with $x \in \supp(w)$. Choose $w^+ \in \supp(u)$ to be the rightmost such element in $u$.
    
    If $w^+$ satisfies (iiia) with $x$, then for the morphism $\delta \colon \fm(\Sigma) \to \N$, where $\delta(x) = 1$ and $\delta(y) = 0$ for all $y \neq x$, then there exists $pk^+ \in R^u_{w,x}$ such that \[\delta(\maxp_w(x)) \leq \delta(\maxp_{pk}(x)).\]
    Since $w^+$ was chosen to be rightmost in $\suff_u(x)$ with $x \in \supp(w)$, for any such $pk^+$ we must have $p = \varepsilon$, as otherwise $|pk|_x = 0$ and $\delta(\maxp_{pk}(x)) = 0$.
    
    Now, let $\xi \colon \fm(\Sigma) \to \N$ be the morphism obtain from Lemma~\ref{lem:ScaleMorphismUp} using $\delta$ and $L > |u|$ and define the sets:
    \begin{align*}
        E &= \{t^+ \in \supp(u) \colon\ x \in \supp(t) \text{ and } \suff_u(w^+) = \suff_u(t^+)\}, \\
        W &= \{t^+ \in E \colon \delta(\maxp_t(x)) \geq \delta(\maxp_s(x)) \text{ for all } s^+ \in E\}, \text{ and}  \\
        W' &= \{t^+ \in E \colon \xi(\maxp_t(x)) \geq \xi(\maxp_s(x)) \text{ for all } s^+ \in E \}.
    \end{align*}

    Then by a similar method to above, one may show that there exists $t^+ \in W$ satisfying (iiib) with $x$.
    
    If instead $w^+$ satisfied (iiib) with $x$, we may proceed by taking $t := w$.
    
    So, let $h^+$ be the element of $\supp(v)$ whose existence is guaranteed by $t^+$ satisfying (iiib) with $x$.
    Then, $x \in \supp(h)$ and $|\suff_v(h^+)|_x = |\suff_u(t^+)|_x = 0$, and so $h^+ \in \supp(\suff_v(x))$. Hence, we have shown the $v$-dual statement to (ii).

    Finally, notice that the $v$-dual of (ii) together with the conditions (i-iv) are $u,v$ symmetric. 
    Thus, by a dual argument, we can show that if $\phi(v)$ has a length $\ell$ retract-free branch $B$, $m$ edges from the end vertex, then $\phi(u)$ has a corresponding branch at least as long.
    Therefore, since we have shown both directions, any branch that exists $m$ edges from the end vertex of one must also exist on the other. We have shown that $\cB(m+1)$ holds. Thus, by induction $\cB(\rho(u)+1)$ holds, that is, $\phi(u) = \phi(v)$ and $u \approx v$ is satisfied by $\flad(a)$.
\end{proof}

Using Proposition \ref{prop:ladradduality}, one may determine dual conditions to those in Theorem~\ref{thm:fladunaryidentities} for classifying enriched identities in $\frad(a)$.

We now briefly cover the identities satisfied by in non-monogenic cases. 
We remark here that the upcoming Proposition \ref{prop:NonMonogenicIdentitiesFlad} and Corollary \ref{cor:fladXidentities} remain true when replacing $\flad(X)$ with the biunary case of $\fad(X)$.

\begin{proposition} \label{prop:NonMonogenicIdentitiesFlad}
    The monoid $\flad(X)$ with $|X| \geq 2$ embeds (as a unary monoid) into $\flad(\{a,b\})$. Consequently, $\flad(X)$ satisfies the same enriched identities as $\flad(\{a,b\})$.
\end{proposition}
\begin{proof}
    Let $X = \{x_i \colon i \in [r-1]\}$. We construct our desired embedding -- one may verify that the morphism $\phi \colon \flad(X) \to \flad(\{a,b\})$ extending $x_i \mapsto ba^ib$ for all $i \in [r-1]$ is suitable.
    
    It follows that any identity satisfied by $\flad(\{a,b\})$ is satisfied by $\flad(X)$. Clearly $\flad(\{a,b\})$ embeds in $\flad(X)$ and hence the converse also holds.
\end{proof}

\begin{corollary}\label{cor:fladXidentities}
    Let $u, v \in \fu(\Sigma)$. Then, the identity $u \approx v$ is satisfied by $\flad(X)$ with $|X| \geq 2$ if and only if $u$ and $v$ represent the same element of $\flad(\Sigma)$.
\end{corollary}
\begin{proof}
    Consider any morphism $\phi \colon \fu(\Sigma) \to \flad(X)$. Then $\phi$ is the unique morphism from $\fu(\Sigma)$ extending the function $\phi|_\Sigma \colon \Sigma \to \flad(X)$.

    \par Now, let $\psi \colon \flad(\Sigma) \to \flad(X)$ be the unique morphism extending $\phi|_\Sigma$ and $\pi \colon \fu(\Sigma) \to \flad(\Sigma)$ be the surjective morphism extending the trivial embedding of $\Sigma$. Note that $\phi$ and $\psi \circ\pi$ agree on $\Sigma$ so, $\phi(x) = \psi(\pi(x))$ for all $x \in \fu(\Sigma)$.

    Clearly, if $\pi(u) = \pi(v)$ then $\phi(u) = \psi(\pi(u)) = \psi(\pi(v)) = \phi(v)$, that is, $u \approx v$ is satisfied by $\flad(X)$.

    Conversely, if $\pi(u) \neq \pi(v)$ then $u \approx v$ is not satisfied by $\flad(\Sigma)$. Note that $\flad(\Sigma)$ embeds in $\flad(X)$ (either since $|\Sigma| = 1$ or by Proposition~\ref{prop:NonMonogenicIdentitiesFlad}), and hence $u \approx v$ is not satisfied by $\flad(X)$. Thus, we are done.
\end{proof}

Thus, by the above corollary, there is an efficient algorithm to check if an identity is satisfied by $\flad(X)$ for $|X| \geq 2$, namely, constructing the corresponding trees for both $u$ and $v$ in $\flad(\Sigma)$ and determining equality. We refer the reader to \cite{KAM4} for a discussion on complexity.

\subsection{Monoid identities}
In this subsection, we classify the monoid identities satisfied by the monoid reducts of $\flad(a)$, $\frad(a)$, and $\fad(a)$.

We say an identity $u \approx v$ is a \emph{monoid identity} if $u,v \in \fm(\Sigma)$ and say a monoid $\cM$ \emph{satisfies} $u \approx v$ if $\phi(u) = \phi(v)$ for all monoid morphisms $\phi \colon \fm(\Sigma) \to \cM$. 

Note that, for $u,v \in \fm(\Sigma)$ and a unary monoid $(M,{}^+)$, the monoid identity $u \approx v$ is satisfied by the monoid reduct $M$ if and only if the identity $u \approx v$ is satisfied by $(M,{}^+)$. Hence, by Theorem~\ref{thm:fladunaryidentities}, we obtain the following corollary.
\begin{corollary} \label{cor:fladidentities}
    Let $u \approx v$ be a monoid identity over $\Sigma$. Then, $u \approx v$ is satisfied by $\flad(a)$ if and only if
    $|u|_y = |v|_y$ and $|\suff_u(x)|_y = |\suff_v(x)|_y$ for all $x,y \in \Sigma$.
\end{corollary}

For $w \in \fm(\Sigma)$ and $x \in \Sigma$, let $\pref_w(x)$ be the shortest prefix of $w$ containing $x$ if $x \in \supp(w)$, and be equal to $w$ otherwise. Now, we get the following corollary for $\frad(a)$ by taking the dual to Corollary~\ref{cor:fladidentities}.
\begin{corollary} \label{cor:fradidentities}
    Let $u \approx v$ be a monoid identity over $\Sigma$. Then, $u \approx v$ is satisfied by $\frad(a)$ if and only if
    $|u|_y = |v|_y$ and $|\pref_u(x)|_y = |\pref_v(x)|_y$ for all $x,y \in \Sigma$.
\end{corollary}

We say a monoid identity $u_1 \approx v_1$ is a \emph{consequence} of a monoid identity $u_2 \approx u_2$ if every monoid that satisfies $u_1 \approx v_1$ also satisfies $u_2 \approx u_2$.
\begin{proposition}[{\cite[Theorem~4.16 \& 4.17]{CMRsylvesteridentities}}]\label{prop:identityconditions}
    A monoid identity $u \approx v$ over $\Sigma$ is a consequence of the monoid identity
    \begin{itemize}
        \item $xyzxty \approx yxzxty$ if and only if $|u|_y = |v|_y$ and $|\suff_u(x)|_y = |\suff_v(x)|_y$ for all $x,y \in \Sigma$,
        \item $xzytxy \approx xzytyx$ if and only if $|u|_y = |v|_y$ and $|\pref_u(x)|_y = |\pref_v(x)|_y$ for all $x,y \in \Sigma$.
    \end{itemize}
\end{proposition}
By the above proposition and Corollaries~\ref{cor:fladidentities} and \ref{cor:fradidentities}, all the monoid identities satisfied by $\flad(a)$ and $\frad(a)$ are consequences of a single monoid identity. Moreover, by \cite[Corollary~4.7]{CMRsylvesteridentities}, we can see that $\flad(a)$ and $\frad(a)$ generate the same monoid varieties as the \textit{sylvester} and \textit{\#-sylvester} monoid, respectively. We refer the reader to \cite[Section 6]{AIR} for other monoids satisfying exactly the same monoid identities as $\flad(a)$.

We close by determining the monoid identities satisfied by $\fad(a)$.

\begin{proposition}
    Let $u \approx v$ be a monoid identity over $\Sigma$. Then $u \approx v$ is satisfied by $\fad(a)$ if and only if $u = v$.
\end{proposition}
\begin{proof}
    Let $\Sigma = \{\sigma_1,\dots,\sigma_n\}$, $w \in \fm(\Sigma)$, and $N_i > |w|$ for $1 \leq i \leq n$ with $N_i \neq N_j$ for all $i \neq j$. Define the morphism $\phi \colon \fm(\Sigma) \to \fad(a)$ by mapping $\sigma_i$ to 
    $(a(a^{N_i})^*)^+a$. The image of $\sigma_i$ under $\phi$ is given below as a retract-free tree where we omit the edge-labelling by $a$.
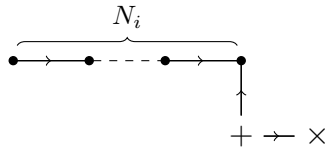
    \[\begin{tikzpicture}
            \setup
            \node (A) at ( 0,0) {\large$+$};
            \Vertex[x=0,y=1]{U}
            \Vertex[x=-1,y=1]{L1}
            \Vertex[x=-2,y=1]{L2}
            \Vertex[x=-3,y=1]{L3}
            \node (R1) at (1,0) {\large$\times$};
            
            \draw[dashed,-] (L2) -- (L1) [dashed] node{};
            \Edge(A)(U)\draw (A) -- (U) node {};
            \Edge(L1)(U)\draw (L1) -- (U) node {};
            \Edge(L3)(L2)\draw (L3) -- (L2) node {};
            \Edge(A)(R1)\draw (A) -- (R1) node {};
            \draw [decorate,decoration={brace,raise=2mm,amplitude=3pt,mirror}] (U) -- (L3) node [font=\small, above, pos=.5, yshift=3mm] {$N_i$};
            
        \end{tikzpicture}\]
As we chose $N_i> |w|$, no branches in $\phi(w)$ can retract. Thus, we are able to reconstruct $w$ from $\phi(w)$, by looking at the length of each branch as $N_i \neq N_j$ for all $i \neq j$.

    \par Therefore, for any monoid identity $u \approx v$, choosing each $N_i > \max(|u|,|v|)$ with $N_i \neq N_j$ for any $i \neq j$, we get that $\phi(u) = \phi(v)$ if and only if $u = v$.
\end{proof}

\bibliographystyle{plain}

\end{document}